\documentclass[11pt,a4paper,openany,oneside]{article}
\usepackage[a4paper,left=2.5cm,right=2.5cm, top=3cm, bottom=3cm]{geometry}
\usepackage{times}
\usepackage{authblk}
\usepackage[latin1]{inputenc}
\usepackage[english]{babel}
\usepackage{upgreek} 
\usepackage{mathrsfs}
\usepackage{stmaryrd}
\usepackage{amssymb,amsmath,amsthm}
\usepackage{hyperref}
\usepackage{graphicx}
\usepackage{epstopdf}
\usepackage{subcaption}
\usepackage{eurosym}
\usepackage{color}
\usepackage{enumerate}
\usepackage{multirow}
\usepackage{mathtools}
\usepackage{scalerel}
\usepackage{booktabs}
\usepackage{float}
\usepackage{amsfonts}
\usepackage{cases}
\usepackage{caption}

\theoremstyle{plain}
\newtheorem{thm}{Theorem}[section]

\newtheorem{lem}[thm]{Lemma}
\newtheorem{lemma}[thm]{Lemma}
\newtheorem{prop}[thm]{Proposition}

\theoremstyle{definition}
\newtheorem{defn}{Definition}[section]
\theoremstyle{remark}
\newtheorem{remark}{Remark}
\definecolor{myred}{rgb}{0.84,0.07,0.14}

\newcommand{\R}{\mathbb R}
\newcommand{\N}{\mathbb N}

\newcommand{\Oo}{\mathcal O}

\newcommand{\les}{\lesssim}

\newcommand{\Omegabar}{\overline\Omega}
\newcommand{\dOmega}{{\partial\Omega}}
\newcommand{\dK}{{\partial K}}
\newcommand{\capp}{{\rm cap}}

\newcommand{\capm}{{\rm cap}_{m,\,\Omega}}
\newcommand{\Capp}{{\rm Cap}_{m,\,\R^N}^\geq}
\newcommand{\capVm}{{\rm cap}_{V^m\!,\,\Omega}}
\newcommand{\capmND}{{\rm cap}_{m,\vartheta,\,\Omega}}
\newcommand{\intOmega}{\int_\Omega}

\newcommand{\eKe}{\varepsilon^{-1}K_\varepsilon}

\newcommand{\cK}{{\mathcal K}}

\newcommand{\cS}{{\mathcal S}}

\newcommand{\cU}{{\mathcal U}}

\newcommand{\cX}{{\mathcal X}}

\renewcommand{\dim}{{\rm dim}\,}

\DeclareFontFamily{U}{mathb}{\hyphenchar\font45}
\DeclareFontShape{U}{mathb}{m}{n}{ <-6> matha5 <6-7> matha6 <7-8>
	mathb7 <8-9> mathb8 <9-10> mathb9 <10-12> mathb10 <12-> mathb12 }{}
\DeclareSymbolFont{mathb}{U}{mathb}{m}{n}

\DeclareMathAccent{\abxring}{0}{mathb}{"38}

\DeclareFontFamily{U}{mathb}{\hyphenchar\font45}
\DeclareFontShape{U}{mathb}{m}{n}{ <-6> matha5 <6-7> matha6 <7-8>
	mathb7 <8-9> mathb8 <9-10> mathb9 <10-12> mathb10 <12-> mathb12 }{}
\DeclareSymbolFont{mathb}{U}{mathb}{m}{n}

\newcommand{\tW}{\widetilde W}
\newcommand{\tWEe}{\tW^E_\varepsilon}

\DeclareMathOperator{\spann}{span}

\DeclareOldFontCommand{\it}{\normalfont\itshape}{\mathit}

\newcommand{\dist}{\text{\rm dist}}
\newcommand{\supp}{\text{\rm supp}}

\newcommand{\iii}[1]{{\left\vert\kern-0.25ex\left\vert\kern-0.25ex\left\vert #1 \right\vert\kern-0.25ex\right\vert\kern-0.25ex\right\vert}}

\numberwithin{equation}{section}

\begin{document}

\renewcommand{\thefootnote}{\fnsymbol{footnote}} 
\renewcommand{\thefootnote}{\fnsymbol{footnote}} 
\renewcommand{\thefootnote}{\arabic{footnote}} 

\title{Perturbed eigenvalues of polyharmonic operators\\ in
domains with small holes
}
\date{October 12, 2022}
\author[a]{Veronica Felli}
\author[b]{Giulio Romani}

\affil[a]{\small Dipartimento di Matematica e Applicazioni\protect\\ Universit\`{a} degli Studi di Milano-Bicocca \protect\\ Via Cozzi 55, 20125 Milano (Italy) \protect\\ \texttt{veronica.felli@unimib.it}\medskip}

\affil[b]{\small Dipartimento di Scienza e Alta Tecnologia \protect\\ Universit\`{a} degli Studi dell'Insubria\protect\\ and\protect\\ RISM-Riemann International School of Mathematics\protect\\ Villa Toeplitz, Via G.B. Vico, 46 - 21100 Varese, Italy \protect\\
\texttt{giulio.romani@uninsubria.it}}

\maketitle


\abstract{We study singular perturbations of eigenvalues of the
  polyharmonic operator on bounded domains under removal of small
    interior compact sets. We consider both homogeneous Dirichlet and
  Navier conditions on the external boundary, while we impose
  homogeneous Dirichlet conditions on the boundary of the removed
  set. To this aim, we develop a notion of capacity which is suitable
  for our higher-order context, and which permits to obtain a
  description of the asymptotic behaviour of perturbed simple
  eigenvalues in terms of a capacity of the removed set, in
  dependence of the respective normalized eigenfunction. Then, in the
  particular case of a subset which is scaling to a point, we apply a
  blow-up analysis to detect the precise convergence rate, which turns
  out to depend on the order of vanishing of the eigenfunction. In
  this respect, an important role is played by Hardy-Rellich
  inequalities in order to identify the appropriate
    functional space containing  the limiting profile. Remarkably,
    for the biharmonic operator this turns out to be the same,
  regardless of the boundary conditions prescribed on the exterior
  boundary.}

\section{Introduction}

The aim of the present work is to study perturbations of the
eigenvalues of the polyharmonic operator $(-\Delta)^m$, $m\geq2$, when
from a given bounded domain $\Omega\subset\R^N$ an interior compact
set $K$ is removed, thus introducing a singular perturbation. We focus
on the case in which $K$ is small, in the sense that its capacity is
asymptotically near $0$, with respect to a notion of capacity suitably
developed for our higher-order setting. More specifically, for
$m\geq2$ we consider the eigenvalue problems
\begin{equation}\label{eq_unperturbed}
	\begin{cases}
		(-\Delta)^mu=\lambda u&\mbox{in}\;\Omega,\\
		u=\partial_nu=\cdots=\partial_n^{m-1}u=0&\mbox{on}\;\dOmega,
	\end{cases}\;{\mbox{resp.}}\;\begin{cases}
		(-\Delta)^mu=\lambda u&\mbox{in}\;\Omega,\\
		u=\Delta u=\dots=\Delta^{m-1} u=0&\mbox{on}\;\dOmega,
	\end{cases}
\end{equation} 
with Dirichlet and Navier boundary conditions (BCs) respectively, and,
given a compact set $K\subset\subset~\!\Omega$, we are interested in
the corresponding eigenvalue problems in case $K$ is removed from
$\Omega$, that is
\begin{equation}\label{eq}
	\begin{cases}
          (-\Delta)^mu=\lambda u&\quad\mbox{in}\;\Omega\setminus K,\\
          u=\partial_nu=\cdots=\partial_n^{m-1}u=0&\quad\mbox{in}\;\partial(\Omega\!\setminus\!K),
	\end{cases}
\end{equation}
in the Dirichlet case, 
and 
\begin{equation}\label{eq_Ke_Nav}
	\begin{cases}
		(-\Delta)^mu=\lambda u&\quad\mbox{in}\;\Omega\setminus K,\\
		u=\Delta u=\dots=\Delta^{m-1} u=0&\quad\mbox{on}\;\dOmega,\\
		u=\partial_nu=\cdots=\partial_n^{m-1}u=0&\quad\mbox{on}\;\partial K,
	\end{cases}
\end{equation}
where, instead, Navier BCs on $\dOmega$ are considered. Note that in
both cases we always deal with Dirichlet BCs on $\dK$.  The goal is to
investigate spectral stability and sharp asymptotic estimates for the
eigenvalues of problems \eqref{eq} and \eqref{eq_Ke_Nav} when $K$
vanishes in a capacitary sense.  \vskip0.2truecm Qualitative
properties of solutions to higher-order problem are deeply related to
the boundary conditions that one prescribes. The most common ones in
the literature are Dirichlet BCs
\begin{equation}\label{DIR_BCs}
	u=\partial_nu=\cdots=\partial_n^{m-1}u=0\qquad\mbox{on}\;\,\dOmega,
\end{equation}
and Navier BCs
\begin{equation}\label{NAV_BCs}
	u=\Delta u=\cdots=\Delta^{m-1}u=0\qquad\mbox{on}\;\,\dOmega.
\end{equation}
Indeed, from the point of view of the applications, they correspond to
the simplest Kirchhoff-Love models of a thin plate, either clamped or
hinged at the boundary, respectively in the Dirichlet and the Navier
case.

While the existence and regularity theory for linear problems is
essentially the same in both cases (see
  e.g. \cite{GGS}), however solutions have relevant differences, even
when $\Omega$ is a smooth domain. The most striking and famous one is
regarding positivity. In the Navier case the solution inherits its
sign from the data, since one can decouple the problem into a system
of second-order equations, for which a maximum principle
holds. Instead, positivity preserving is in general lost in the
Dirichlet case, even for smooth and convex domains, except for
peculiar situations in which one can rely on a global analysis of the
Green function, such as for the case of the ball and its smooth
deformations, see \cite{GGS,GR}. On the other hand, functions which
undergo Dirichlet BCs can be trivially extended by $0$ outside the
domain, so that the extension continues to belong to the same
higher-order Sobolev space, while this is not true anymore for
solutions of Navier problems because of possible jumps on $\dOmega$ of
the normal derivative. Motivated by these arguments, we investigate
perturbations of the eigenvalues of $(-\Delta)^m$ in both context of
Dirichlet and Navier BCs on $\dOmega$. Since we rely more on extension
properties rather than positivity issues,
our
analysis will be harder in the Navier~case.
\vskip0.2truecm

In the second-order case (i.e. $m=1$), spectral stability under
removal of small (condenser) capacity sets is proved in
\cite{Courtois} in a very general context, see also
\cite{Bertrand-Colbois} and \cite{flucher}. More specifically, in
\cite{Courtois} it is shown there that the function
$\lambda(\Omega\setminus K)-\lambda(\Omega)$ is differentiable with
respect to the capacity of the removed set $K$ relative to $\Omega$. A
sharp quantification of the vanishing order of the variation of simple
eigenvalues is given in \cite{AFHL}, when concentrating families of
compact sets are considered: the precise rate of convergence is
asymptotic to the $u$-capacity associated to the limit eigenfunction
$u$ (see \cite[Definition 2.1]{Bertrand-Colbois} and
\cite[(14)]{Courtois} for the notion of $u$-capacity) and sharp
asymptotic estimates are given in terms of the diameter of the removed
set if the limit set is a point in $\R^2$, and either if the
eigenfunction does not vanish there, or in case of specific
concentrating sets such as disks or segments.  Asymptotic estimates of
$u$-capacities and eigenvalues of the Dirichlet Laplacian, on bounded
planar domains with small holes of the more general form
$\varepsilon\omega$ with $\varepsilon\to0$, are given in \cite{ABNLM}.
In both \cite{AFHL} and \cite{ABNLM}, a tool that helps to provide
precise asymptotic estimates in dimension two is given by elliptic
coordinates, which allow rewriting the equations satisfied by the
capacitary potentials in a rather explicit way and which however do
not have a simple analogue in higher dimensions.  In the complementary
case $N\geq3$, an approach based on a blow-up argument is used in
\cite{FNO} to derive sharp asymptotic estimates of the $u$-capacity,
and consequently of the eigenvalue variation, for general families of
sets which may also concentrate at the boundary. This method has been
applied also to fractional problems in \cite{AFN}.

Higher-order problems $m\geq2$ are studied in \cite{CN,KLW,LWK}, where
asymptotic expansions of eigenvalues of biharmonic operators are
obtained under removal of a family of sets which are uniformly
vanishing to a point $\{x_0\}$. All these papers deal with the
two-dimensional case and only Dirichlet boundary conditions, both on
$\dOmega$ and on $\dK$, are considered. The main difference with the
corresponding two-dimensional second-order problem, is that the
limiting problem involves the punctured domain
$\Omega\setminus\{x_0\}$. In \cite{CN} formal recursive asymptotic
expansions are found in the nondegenerate case, namely when the
gradient of the corresponding eigenfunction does not vanish at $x_0$,
as well as in the degenerate case. In the former case, these
expansions are justified in a suitable functional setting which makes
use of weighted Sobolev spaces, named after Kondrat'ev, in order to
deal with the point constraint. On the other hand, motivated by the
study of MEMS-devices, in \cite{KLW}, the asymptotic behaviour of
eigenpairs is formally obtained, in both nondegenerate and degenerate
cases, using the method of matching asymptotic expansions. A more
delicate situation is taken into account in \cite{LWK}, when both the
removed subdomain is vanishing, as well as the biharmonic part of the
operator, provided a second-order term is introduced in the
equation. In all these works, the asymptotic expansions of the
perturbed eigenvalues are of logarithmic kind, fact that recalls the
expansion in the two-dimensional case for the Laplace operator given
in \cite[Theorem 1.7]{AFHL}. We note however that, unlike what happens
for the second order problem, capacities cannot play there the role of
perturbation parameters, since in dimension $2$ the higher order
capacity of a point (defined as in \eqref{capVm}) is different from
zero; this is also the reason why the limiting problem is formulated
in the punctured domain.  We mention that the spectral behavior of
higher order elliptic operators upon domain perturbation is
investigated also in \cite{AL} for Dirichlet, Neumann and intermediate
boundary conditions.

The first aim of
the present paper is a rigorous description of the asymptotic
behaviour of the perturbed eigenvalues for polyharmonic operators
$(-\Delta)^m$  for any $m\geq2$ and for a large class of removed sets,
in the spirit of \cite{AFHL,FNO,AFN}. Since we deal with sets of
vanishing capacities, we are focused on the high dimensional
case $N\geq2m$. Furthermore, as second important objective,
we 
investigate whether and how different boundary conditions on $\dOmega$
affect the analysis. As already remarked, in the present work
we consider Dirichlet boundary conditions on $\dK$. In order to have a complete
picture of the influence of the boundary conditions, the complementary
situation of Navier BCs on $\dK$ should be addressed. However, the
techniques developed in the present work strongly rely
on extension properties which are characteristic of the Dirichlet
case, so that a different approach should be devised to treat
the Navier case on $\dK$. We plan to address this in a future work.
\vskip0.2truecm

In order to give the precise statements of the main results, we first
describe the functional setting and the notation we are going to use
throughout the paper.

\paragraph{Notation} We denote the normal derivative of the function
$u$ by $\partial_nu$. For a set $D\subset\R^N$, $\cU(D)$ denotes some
open neighbourhood of $D$, $C^\infty_0(D)$ is the space of the
infinitely differentiable functions which are compactly supported in
$D$, and $L^p(D)$ with $p\in[1,+\infty]$ is the space of
$p$-integrable functions. The norm of $L^p(D)$ is denoted simply by
$\|\cdot\|_p$ whenever the domain is clear from the context.  For
every $m\in \N$ and $u:D\to\R$ with $D\subset\R^N$, we denote as
$D^mu$ the tensor of $m$-th order derivatives of $u$ and define
$|D^mu|^2=\sum_{|\alpha|=m}|D^\alpha u|^2$, where $|\alpha|$ is the
length of the multi-index $\alpha$.

The symbol $\les$ is
used when an inequality is true up to an omitted structural constant,
and we write $f=\Oo(g)$ (resp. $f={\scriptstyle \mathcal{O}}(g))$ as
$x\to x_0$ when there exists a constant $C>0$ such that
$|f(x)|\leq C|g(x)|$ in a neighbourhood of $x_0$
(resp. $\frac{f(x)}{g(x)}\to0$ as $x\to x_0$).

\subsection{The functional setting}\label{Sec_spaces}
Let $\Omega$ be a bounded smooth domain in $\R^N$. In order to treat
at once different boundary conditions on $\dOmega$, i.e. the settings
of problems \eqref{eq} and \eqref{eq_Ke_Nav}, we introduce the
following notation. For $m\geq2$ the set
$V^m(\Omega)\subset H^m(\Omega)$ is defined either as
\begin{equation*}
	V^m(\Omega):=H^m_0(\Omega)
      \end{equation*}
      in case Dirichlet boundary conditions \eqref{DIR_BCs} are
      prescribed on $\dOmega$, where $H^m_0(\Omega)$ is the closure in
      $H^m(\Omega)$ of $C^\infty_0(\Omega)$, or by
\begin{equation*}
	V^m(\Omega):=H^m_\vartheta(\Omega)
\end{equation*}
if Navier boundary conditions \eqref{NAV_BCs} are assumed on
$\dOmega$. Here $H^m_\vartheta(\Omega)$ is the closure in
$H^m(\Omega)$ of the space
\begin{equation*}
  C^m_\vartheta(\Omegabar):=\left\{u\in C^m(\Omegabar)\ \big|\
    \Delta^ju|_{\dOmega}
    =0\,\mbox{for all}\,0\leq j<\tfrac m2\right\}
\end{equation*}
and it can be characterized as
\begin{equation*}
  H^m_\vartheta(\Omega)=\left\{u\in H^m(\Omega)\ \big|\
    \Delta^ju|_{\dOmega}=0\
    \mbox{in the sense of traces}\ \mbox{for all}\ 0\leq j<\tfrac m2\right\}.
\end{equation*}
Note that for $m=2$ we have
$H^2_\vartheta(\Omega)=H^2(\Omega)\cap H^1_0(\Omega)$. In both cases
$V^m(\Omega)$ is a closed subspace of $H^m(\Omega)$. Moreover, for a
bounded domain $\Omega\subset\R^N$, the norms
$$\|\cdot\|_{H^m(\Omega)}:=\sum_{|\alpha|\leq m}\|D^\alpha\cdot\|_{L^2(\Omega)}$$
(with the multi-index notation) and
$$\|\nabla^m\cdot\|_{L^2(\Omega)},\qquad\mbox{where}\quad\nabla^mf:=\begin{cases}
	\Delta^{\frac m2}f&\quad\mbox{for}\;m\;\mbox{even},\\
	\nabla\Delta^{\frac {m-1}2}f&\quad\mbox{for}\;m\;\mbox{odd},
\end{cases}$$
are equivalent on both $H^m_0(\Omega)$ and $H^m_\vartheta(\Omega)$,
see e.g. \cite[Theorem 2.2]{GGS} for the Dirichlet case and
\cite{GGS_art} for the Navier case. In particular,
  there exists a positive constant $C=C(N,m,\Omega)>0$, depending only
  on $N$, $m$, and $\Omega$, such that
  \begin{equation}\label{eq:equiv-norme}
    \|u\|_{H^m(\Omega)}\leq
    C\|\nabla^m u\|_{L^2(\Omega)}\quad\text{for all }u\in V^m(\Omega).
  \end{equation}
Note also that in the Dirichlet case all boundary conditions are
stable, and therefore they are all included in the definition of the
space $H^m_0(\Omega)$; on the other hand, only the first half of the
Navier conditions are stable, while the boundary conditions
$\Delta^ju|_{\dOmega}=0$ for $\tfrac m2\leq j\leq m-1$ are natural and
thus do not appear in the definition of $H^m_\vartheta(\Omega)$. For a
comprehensive discussion, see \cite[Sec.2.4]{GGS}.
\vskip0.2truecm
The next spaces are relevant when a ``hole'' is produced in the
domain. For a compact set $K\subset\Omega$, we define
$$V^m_0(\Omega\setminus K):=\begin{cases}
	H^m_0(\Omega\setminus K)&\quad\mbox{in the Dirichlet case},\\
	H^m_{\vartheta,0}(\Omega\setminus K)&\quad\mbox{in the Navier case}.
\end{cases}$$
Here $H^m_{\vartheta,0}(\Omega\setminus K)$ denotes the space suitable
for Navier BCs on $\dOmega$ and Dirichlet BCs on $\dK$. More
precisely,
$H^m_{\vartheta,0}(\Omega\setminus K)$ is the closure 
in   $H^m_{\vartheta}(\Omega)$  of
\begin{equation*}
  C^m_{\vartheta,0}(\Omegabar\setminus K):=\left\{u\in C^m_\vartheta(\Omegabar)\ \big|\ \supp\, u\cap\cU(K)=\emptyset\ \mbox{for some}\ \cU(K)\right\}.
\end{equation*}
In case $\partial K$ is smooth, $u\in
H^m_{\vartheta,0}(\Omega\setminus K)$ if and only if $u\in
H^m(\Omega\setminus K)$ and 
\begin{equation*}
	\Delta^ju|_{\dOmega}=0\ \,\mbox{for all}\ \,0\leq j<\tfrac m2\quad\,\mbox{and}\quad\,\partial_n^hu|_{\dK}=0\ \,\mbox{for all}\ \,0\leq h\leq m-1
\end{equation*}
in the sense of $L^2$-traces. Note that we have the following chain of inclusions
\begin{equation}\label{inclusion}
	H^m_0(\Omega\setminus K)\subsetneq H^m_{\vartheta,0}(\Omega\setminus K)\subsetneq H^m_\vartheta(\Omega)\subsetneq H^m(\Omega),
\end{equation}
where the second inclusion holds by extending to $0$ in $K$ functions
defined in $\Omega\setminus K$, thanks to the Dirichlet conditions
imposed on $\dK$. For the same reason, note also that, for any compact
sets $K_1,K_2$ such that 
$K_1\subset K_2\subset\Omega$, one has
\begin{equation*}
	V^m(\Omega\setminus K_2)\subset V^m(\Omega\setminus K_1).
\end{equation*}

\vskip0.2truecm

All such spaces are Hilbert spaces with scalar product\footnote{We
  always omit to indicate the scalar product in $\R^N$ with $\cdot$.}
$q_m(u,v):=\intOmega\nabla^mu\,\nabla^mv$. Note that, unlike the
general case, $q_m(\cdot,\cdot)$ does not involve boundary integrals,
see \cite[Sec.2.4]{GGS}. By standard arguments \cite[Theorem
2.15]{GGS}, the linear problem $(-\Delta)^mu=f$ in $\Omega\setminus K$,
with $f\in L^2(\Omega\setminus K)$ and boundary conditions either
  \eqref{DIR_BCs} or \eqref{NAV_BCs}, admits a unique weak solution
$u\in V^m_0(\Omega\setminus K)$, in the sense that
\begin{equation*}
	\int_\Omega\nabla^mu\,\nabla^m\varphi=\intOmega
        f\varphi\qquad\mbox{for all}\ \
        \varphi\in V^m_0(\Omega\setminus K).
\end{equation*}
Analogously, we define the eigenvalues of problems \eqref{eq} and
\eqref{eq_Ke_Nav} in the weak sense. We say that $(\lambda,u)$ is an
eigenpair of \eqref{eq} (resp. \eqref{eq_Ke_Nav}) if
$(\lambda,u)\in\R\times V^m_0(\Omega\setminus K)$ satisfies
\begin{equation}\label{weak_sol_eigfct}
  u\not\equiv0\quad\text{and}\quad
  \int_\Omega\nabla^mu\,\nabla^m\varphi=
  \lambda\intOmega u\varphi\qquad\mbox{for all}\;\varphi\in V^m_0(\Omega\setminus K).
\end{equation}
By classical spectral theory, problems \eqref{eq} and
\eqref{eq_Ke_Nav} admit a diverging sequence of positive eigenvalues
$$0<\lambda_1(\Omega\setminus K)\leq\cdots\leq\lambda_j(\Omega\setminus K)\leq\cdots\to+\infty,$$
where each one is repeated as many times as its multiplicity. Of
course the same holds for the unperturbed problems
\eqref{eq_unperturbed}, whose eigenvalues are denoted as
$\left(\lambda_j(\Omega)\right)_{j\in\N}$. We recall that the
eigenvalues may be variationally characterized as
\begin{equation}\label{var_char_eigv}
  \lambda_j(\Omega\setminus K)=\min_{\substack{\cX_j\subset
      V^m_0(\Omega\setminus K) \\
      \dim\cX_j=j}}\,\max_{v\in\cX_j}
  \frac{\int_{\Omega\setminus K}|\nabla^mv|^2}{\int_{\Omega\setminus K}|v|^2}\,.
\end{equation}
Finally, for $\Omega$ and $K$ as before, we define
	\begin{equation}\label{Xm}
		X^m(\Omega):=\begin{cases}
			C^\infty_0(\Omega)&\quad\mbox{in the Dirichlet case},\\
			C^m_\vartheta(\Omegabar)&\quad\mbox{in the Navier case}
		\end{cases}
\end{equation}
and
\begin{equation*}
	X^m_0(\Omega\setminus K):=\begin{cases}
		C^\infty_0(\Omega\setminus K)&\quad\mbox{in the Dirichlet case},\\
		C^m_{\vartheta,0}(\Omegabar\setminus K)&\quad\mbox{in the Navier case},
	\end{cases}
\end{equation*}
for the sake of a compact notation in some of the proofs.

\subsection{Main results}

In the spirit of the previously cited works \cite{AFHL,AFN,FNO},
asymptotic expansions of eigenvalues under removal of small sets can
be established treating as a perturbation parameter a suitable notion
of capacity. Extending to the higher-order Sobolev framework  the classical definition in the second-order
case, for  every compact set $K\subset\Omega$ we define the
(condenser) $V^m$-\textit{capacity} of $K$ in $\Omega$ as
\begin{equation}\label{capVm}
	\capVm(K):=\inf\left\{\int_\Omega|\nabla^mf|^2\,\Big|\,f\in V^m(\Omega),\, f-\eta_K\in V^m_0(\Omega\setminus K)\right\},
\end{equation}
where $\eta_K$ is a fixed smooth function such that
$\supp\,\eta_K\subset\Omega$ and $\eta_K\equiv1$ in a neighbourhood of
$K$. The $V^m$-capacity of a set $K$ gives an indication about its
relevance for the higher-order Sobolev space $V^m$, in the sense that
zero $V^m$-capacity sets do not affect the space $V^m(\Omega)$ when
they are removed from $\Omega$, and hence nor the spectrum of the
polyharmonic operator (Proposition \ref{Capacity_zero}).
\vskip0.2truecm

In our analysis, a notion of ``weighted'' capacity, which represents
the higher order analogue of the $u$-capacity introduced in
\cite[Definition 2.1]{Bertrand-Colbois} and \cite[(14)]{Courtois} for
second order problems, will be significant too. Given a function
$u\in V^m(\Omega)$, we define the $(u,V^m)$-\textit{capacity} of $K$
in $\Omega$ as
\begin{equation}\label{capVmu}
  \capVm(K,u):=\inf\left\{\int_\Omega|\nabla^mf|^2\,\Big|\,f\in
    V^m(\Omega),\,
    f-u\in V^m_0(\Omega\setminus K)\right\}.
\end{equation}
Note that $u$ is relevant only in a neighbourhood of $K$. Hence,
$\capVm(K,u)=\capVm(K,\eta_Ku)$ for any cut-off function $\eta_K$ as
before. This permits to extend the notion of $(u,V^m)$-capacity to
functions $u\in H^m_{loc}(\R^N)$.

\vskip0.2truecm For those cases in which we need to distinguish the
capacities according to the boundary conditions on $\dOmega$, we use
the following notation:
\begin{equation*}
\capm(K):=\capp_{H^m_0,\Omega}(K)\qquad\mbox{and}\qquad\capmND(K):=\capp_{H^m_\vartheta,\Omega}(K),
\end{equation*}
for the Dirichlet and Navier BCs on $\dOmega$, respectively. Similarly we denote
\begin{equation}\label{eq:notcapm}
		\capm(K,u):=\capp_{H^m_0,\Omega}(K,u)\qquad\mbox{and}\qquad\capmND(K,u):=\capp_{H^m_\vartheta,\Omega}(K,u).
\end{equation}
We point out that the $V^m$-capacity as well as the $(u,V^m)$-capacity
of a compact set $K$ are attained by a unique minimizer, which is
called capacitary potential, and which we denote by $W_K$ and
$W_{K,u}$ respectively. The proof of the attainment of both
capacities, together with some basic properties which will be used
throughout the paper, is presented in Section \ref{Section_cap}.

Our first result is about the stability of the spectrum of
$(-\Delta)^m$, once a set of small $V^m$-capacity is removed.
\begin{thm}\label{Alternative_conv_eigv}
	Let $N\geq2m$ and $\Omega\subset\R^N$ be a smooth bounded domain. Suppose one of the following:
	\begin{enumerate}
		\item[(D)] $V^m(\Omega)=H^m_0(\Omega)$ and
                  $K\subset\Omega$ is compact;
		\item[(N)] $V^m(\Omega)=H^m_\vartheta(\Omega)$ and the
                  exists $K_0\subset\Omega$ compact such that $K$ is
                  compact and $K\subset K_0$.
	\end{enumerate}
	Denote by $\lambda_j(\Omega)$ and $\lambda_j(\Omega\setminus
        K)$,
        $j\in\N\setminus\{0\}$, the eigenvalues respectively for
        \eqref{eq_unperturbed} and \eqref{weak_sol_eigfct}. For all
        $j\in\N\setminus\{0\}$, there exist $\delta>0$ and $C>0$ (which depends on
        $K_0$ in the Navier case (N)) such that, if
        $\capVm(K)<\delta$, one has
	\begin{equation*}
		|\lambda_j(\Omega\setminus K)-\lambda_j(\Omega)|\leq C\left(\capVm(K)\right)^{1/2}.
	\end{equation*}
	In particular $\lambda_j(\Omega\setminus K)\to\lambda_j(\Omega)$ as $\capVm(K)\to0$.
\end{thm}
The proof of Theorem \ref{Alternative_conv_eigv} is based on the variational characterization of the eigenvalues \eqref{var_char_eigv} and it is detailed in Section \ref{Section_spectral_stability}. We remark that spectral stability in a more general higher-order context was also established in \cite{AL} with a different approach. Here we propose a self-contained and simple proof for our Dirichlet and Navier-Dirichlet settings.

Aiming now at a more precise estimate of the convergence rate, we introduce the following notion of convergence of sets.

\begin{defn}\label{def_conv}
	Let $\{K_\varepsilon\}_{\varepsilon>0}$ be a family of compact sets contained in $\Omega$. We say that $K_\varepsilon$ \textit{is concentrating to a compact set} $K\subset\Omega$ as $\varepsilon \to0$ if, for every open set $U\subseteq\Omega$ such that $U\supset K$, there exists $\varepsilon_U>0$ such that $U\supset K_\varepsilon$ for every $\varepsilon\in(0,\varepsilon_U)$.
\end{defn}
An example is given by a decreasing family of compact sets, see e.g. \cite[Example 3.7]{FNO}. Note that this property alone is not sufficient to have the standard (i.e. metric) convergence of sets. For instance, the uniqueness of the limit set is not assured (e.g. if $K_\varepsilon$ is concentrating to $K$ then $K_\varepsilon$ is concentrating also to $\widetilde K$ for any compact set $\widetilde K$ which contains $K$). However, as for second-order problems, in the case of a $0$-capacity limit set, this concept of convergence of sets is enough to prove the continuity of the capacity (Proposition \ref{cap_pot_stability}) and the Mosco convergence \cite{Da,Mo} of the respective $V^m$-spaces (Proposition \ref{Mosco}). These will be the tools needed for a sharp asymptotic expansion of a perturbed simple eigenvalue $\lambda_J(\Omega\setminus K_\varepsilon)$ in terms of the $(u_J,V^m)$-capacity of the vanishing compact sets $K_\varepsilon$, where $u_J$ is a normalized eigenfunction relative to $\lambda_J(\Omega)$.
\begin{thm}\label{Prop_eigv_cap}
Let $N\geq2m$ and $\Omega\subset\R^N$ be a smooth bounded domain.  Let $\lambda_J(\Omega)$ be a simple eigenvalue of \eqref{eq_unperturbed} and $u_J\in V^m(\Omega)$ be a corresponding eigenfunction normalized in $L^2(\Omega)$. Let $\{K_\varepsilon\}_{\varepsilon>0}$ be a family of compact sets concentrating, as $\varepsilon \to0$, to a compact set $K$ with $\capVm(K)=0$. Then, as $\varepsilon \to0$,
	\begin{equation}\label{formula_asympt}
		\lambda_J(\Omega\setminus K_\varepsilon)=\lambda_J(\Omega)+\capVm(K_\varepsilon,u_J)+{\scriptstyle \mathcal{O}}(\capVm(K_\varepsilon,u_J)).
	\end{equation}
\end{thm}

Theorem \ref{Prop_eigv_cap} is the higher-order counterpart of
\cite[Theorem 1.4]{AFHL} and its proof is presented in Section
\ref{Section_asympt}. In the expansion \eqref{formula_asympt}, the
asymptotic parameter is the $(u_J,V^m)$-capacity of the vanishing
set. The next aim is to quantify $\capVm(K_\varepsilon,u_J)$ as a
function of the diameter of $K_\varepsilon$. In this respect, we focus
on the particular case in which the limit set $K$ is a point
$x_0\in\Omega$  (which in dimension $N\geq2m$ has zero $V^m$-capacity,
see Proposition \ref{Capacity_point}); without loss of generality, we
consider $x_0=0$. We deal with a uniformly
shrinking family of compact sets $K_\varepsilon$, the model case being
$K_\varepsilon=\varepsilon\cK\ni0$ for some fixed compact set
$\cK\subset\R^N$. In this case, assuming $0$ to be an interior point
of $\Omega$, and having the operator $(-\Delta)^m-\lambda$ constant
coefficients, the eigenfunction $u_J$ is analytic at $0$, see
\cite{John}, and hence it does not have infinite order of vanishing there. Therefore, 
there exist $\gamma\in\N$ and a $\gamma$-homogeneous polyharmonic
polynomial $U_0\in H^m_{loc}(\R^N)$ such that
\begin{equation}\label{claim_U}
	U_\varepsilon:=\frac{u_J(\varepsilon\,\cdot)}{\varepsilon^\gamma}\to U_0\qquad\mbox{in}\;H^m(B_R(0))
\end{equation}
for all $R>0$ as $\varepsilon\to 0$. This fact follows from a general result about elliptic equations by Bers \cite[Sec.4 Theorem 1]{Bers}, see also \cite[Theorem 2.1]{Chen}, provided - as in our case - one discards the possibility of an infinite order of vanishing.

In light of \eqref{claim_U}, our strategy to find an asymptotic
expansion of $\capVm(K_\varepsilon,u_J)$ is based on a blow-up
argument: we rescale the boundary value problem defining the
capacitary potential $W_{K_\varepsilon,u_J}$, find a limit equation on
$\R^N\setminus\cK$, and prove the convergence of the family of
rescaled capacitary potentials to the one for the limiting problem. To
this aim, a suitable notion of capacity in $\R^N$, involving
homogeneous higher-order Sobolev spaces $D^{m,2}_0(\R^N)$ and denoted
by $\capp_{m,\R^N}$, will be needed, see Section \ref{Sec_Cap_RN}.
The asymptotic expansion of $\capVm(K_\varepsilon,u_J)$ obtained by
these arguments turns out to depend on the order of vanishing of $u_J$
at the point $0$. More precisely, we have the following results, which
we state below for the model case $K_\varepsilon=\varepsilon\cK$ and
prove in more generality in Section \ref{Sec_blowup}. For the
Dirichlet case we have the following:
\begin{thm}[Dirichlet case]\label{Thm_blowup_easycase}
  Let $N>2m$ and $\Omega\subset\R^N$ be a bounded smooth domain with
  $0\in\Omega$. Let $\cK\subset\R^N$ be a fixed compact set and, for
  all $\varepsilon>0$, $K_\varepsilon=\varepsilon\cK$. Let $\lambda_J$
  be an eigenvalue of \eqref{eq_unperturbed} with Dirichlet boundary
  conditions and $u_J\in H^m_0(\Omega)$ be a corresponding
  eigenfunction normalized in $L^2(\Omega)$. Then
	\begin{equation}\label{Thm_blowup_asympt_easycase}
          {\rm cap}_{m\!,\,\Omega}(K_\varepsilon,u_J)
          =\varepsilon^{N-2m+2\gamma}\left(\capp_{m,\R^N}(\cK,U_0)+{\scriptstyle \mathcal{O}}(1)\right)
	\end{equation}
	as $\varepsilon\to0$, with $\gamma$ and $U_0$ as in \eqref{claim_U}.
\end{thm}
The dimensional restriction $N>2m$ is mainly due to the possibility of
characterizing higher-order homogeneous Sobolev spaces as concrete
functional spaces satisfying Sobolev and Hardy-type inequalities (see
Sections \ref{sec:homog-sobol-spac} and \ref{Sec_Hardy}). In the
conformal case $N=2m$ such spaces are instead made of classes of
functions defined up to additive polynomials, see
\cite[II.6-7]{Galdi}. In the Navier setting, we need to restrict to
the biharmonic case $m=2$.
\begin{thm}[Navier case]\label{Thm_blowup_easycase_Nav}
  Let $N>4$ and $\Omega\subset\R^N$ be a bounded smooth domain with
  $0\in\Omega$. Let $\cK\subset\R^N$ be a fixed compact set and, for
  all $\varepsilon>0$, $K_\varepsilon=\varepsilon\cK$. Let $\lambda_J$
  be an eigenvalue of \eqref{eq_unperturbed} with Navier boundary
  conditions and $u_J\in H^2_\vartheta(\Omega)$ be a corresponding
  eigenfunction normalized in $L^2(\Omega)$. Then
	\begin{equation*}
		{\rm cap}_{2,\vartheta\!,\,\Omega}(K_\varepsilon,u_J)=\varepsilon^{N-4+2\gamma}\left(\capp_{2,\R^N}(\cK,U_0)+{\scriptstyle \mathcal{O}}(1)\right)
	\end{equation*}
	as $\varepsilon\to0$, with $\gamma$ and $U_0$ as in \eqref{claim_U} with $m=2$.
\end{thm}
It is remarkable that the same asymptotic expansion
\eqref{Thm_blowup_asympt_easycase} for $m=2$ holds true for both
Dirichlet and Navier BCs on $\dOmega$. As a consequence, imposing
different conditions on the external boundary does not affect the
first term of the asymptotic expansion of the perturbed
eigenvalues. In the proof of Theorems \ref{Thm_blowup_easycase} and
\ref{Thm_blowup_easycase_Nav} we will need to distinguish between the
two settings. If in the case of Dirichlet BCs on $\dOmega$ the natural
candidate as functional space for the limiting problem is
$D^{m,2}_0(\R^N\setminus\cK)$, on the other hand, in the Navier case,
because of the impracticability of the trivial extension of a function
outside $\Omega$, this is not evident and follows after a more
involved analysis which makes use of suitable Hardy-Rellich
inequalities. In Section \ref{Sec_Hardy} we give the precise statement
and proofs of such inequalities. This is the main reason for the
restriction to the case $m=2$, see Section \ref{Sec_blowup}.

Braiding together Theorem \ref{Prop_eigv_cap} and Theorems
\ref{Thm_blowup_easycase}-\ref{Thm_blowup_easycase_Nav}, we obtain the
following sharp asymptotic expansions of
$\lambda_J(\Omega\setminus K_\varepsilon)$, stated here for the model
case $K_\varepsilon=\varepsilon\cK$.
\begin{thm}[Dirichlet case]\label{asymptotic_eigv_easycase}
  Let $N>2m$ and $\Omega\subset\R^N$ be a bounded smooth domain
  containing $0$. Let $\cK\subset\R^N$ be a fixed compact set and, for
  all $\varepsilon>0$, $K_\varepsilon=\varepsilon\cK$. Let $\lambda_J$
  be a simple eigenvalue of \eqref{eq_unperturbed} with Dirichlet
  boundary conditions and let $u_J\in H^m_0(\Omega)$ be a
  corresponding eigenfunction normalized in $L^2(\Omega)$. Then
	\begin{equation*}
          \lambda_J(\Omega\setminus K_\varepsilon)=\lambda_J(\Omega)
          +\varepsilon^{N-2m+2\gamma}\left(\capp_{m,\R^N}(\cK,U_0)+{\scriptstyle \mathcal{O}}(1)\right)
	\end{equation*}
	as $\varepsilon\to0$, with $\gamma$ and $U_0$ as in \eqref{claim_U}.
\end{thm}
\begin{thm}[Navier case]\label{asymptotic_eigv_easycase_Nav}
  Let $N>4$ and $\Omega\subset\R^N$ be a bounded smooth domain
  containing $0$. Let $\cK\subset\R^N$ be a fixed compact set and, for
  all $\varepsilon>0$, $K_\varepsilon=\varepsilon\cK$. Let $\lambda_J$
  be a simple eigenvalue of \eqref{eq_unperturbed} with Navier
  boundary conditions and let $u_J\in H^2_\vartheta(\Omega)$ be a
  corresponding eigenfunction normalized in $L^2(\Omega)$. Then
	\begin{equation*}
          \lambda_J(\Omega\setminus K_\varepsilon)
          =\lambda_J(\Omega)+\varepsilon^{N-4+2\gamma}\left(\capp_{2,\R^N}(\cK,U_0)+{\scriptstyle \mathcal{O}}(1)\right)
	\end{equation*}
	as $\varepsilon\to0$, with $\gamma$ and $U_0$ as in \eqref{claim_U} with $m=2$.
\end{thm}

Theorems \ref{Thm_blowup_easycase}-\ref{asymptotic_eigv_easycase_Nav}
deal with the model case $K_\varepsilon=\varepsilon\cK$. Section
\ref{Sec_blowup} will be devoted to the proof of their analogues for a
more comprehensive setting of general families of concentrating
compact sets $\{K_\varepsilon\}_{\varepsilon>0}$ which uniformly
shrink to a point, see Theorems
\ref{Thm_blowup}-\ref{asymptotic_exp_eigv_cap_Nav}.

The asymptotic expansion provided by Theorems
\ref{asymptotic_eigv_easycase}-\ref{asymptotic_eigv_easycase_Nav}
detects the sharp vanishing rate of the eigenvalue variation whenever
$\capp_{m,\R^N}(\cK,U_0)\neq0$. In Section \ref{Sec_suffcond} we
establish sufficient conditions for this to hold. In particular, this
will always be the case when the Lebesgue measure of $\cK$ is positive
(Proposition \ref{sharp_asympt}), or when either the eigenfunction
$u_J$ does not vanish at the point $x_0$ (Proposition
\ref{suff_cond_nonvanishing}) or it does vanish but the compactum
$\cK$ and the null-set of the limiting polynomial $U_0$ in
\eqref{claim_U} are ``transversal enough'' (Proposition
\ref{Prop_transv}).

The paper is then concluded by the short Section \ref{Sec_OP} which
contains a discussion about questions which are left open by our
analysis and possible directions in which our results may be extended.

\section{Definition of higher-order capacity with Dirichlet and Navier
  BCs and first properties}\label{Section_cap}
The aim of this section is to introduce a notion of \textit{capacity}
which agrees with the higher-order framework of the problem and which
turns out to be an important tool in order to study the asymptotics of
the eigenvalues of the perturbed problems
\eqref{eq}-\eqref{eq_Ke_Nav}. The concept of (condenser) capacity,
well-known for the second-order case, was first considered in the
higher-order setting by Maz'ya for bounded domains on which Dirichlet
boundary conditions are imposed, or for the whole space,
see\footnote{In these works the higher-order capacity is defined
  through the $L^p$-norm of the tensor of the $m$-th derivatives
  $D^mu$. However, the two norms are equivalent on any bounded smooth
  domain.} e.g. \cite{M1,M}. In Section \ref{Sec_cap_Omega} we propose
an unified treatment for both Dirichlet and Navier settings and
establish the main properties of the capacities defined by
\eqref{capVm}-\eqref{capVmu}. In Section \ref{Sec_Cap_RN} we recall
the main properties of the homogeneous Sobolev spaces and establish a
Hardy-Rellich inequality with intermediate derivatives. Moreover  we
introduce 
the notion of capacity of a compact set in the whole space $\R^N$ for
large dimensions $N>2m$.

\subsection{Higher-order capacities in $\boldsymbol{V^m_0}$}\label{Sec_cap_Omega}

Let $m\in\N\setminus\{0\}$, $\Omega$ be a bounded smooth domain in
$\R^N$ and $K$ be a compact subset of $\Omega$. First, we observe that
both capacities \eqref{capVm}-\eqref{capVmu} are attained. Indeed, for
any $u\in V^m(\Omega)$, we have that
$S_u:=\left\{g\in V^m(\Omega)\,|\,g-u\in V^m_0(\Omega\setminus
  K)\right\}$ is an affine hyperplane in $V^m(\Omega)$, so in
particular a convex set. This implies that there exists a unique
element in $V^m(\Omega)$ which minimizes the distance from the origin,
i.e. the norm $\|\nabla^m \cdot\|_2$ in $S_u$, which is
called \textit{capacitary potential} and is denoted by $W_{K,u}$
(in case $u$ is replaced by $\eta_K$, we simply denote it by
$W_K$). This means that $W_{K,u}$ is such that
\begin{equation*}
	\capVm(K,u)=\int_\Omega|\nabla^mW_{K,u}|^2
\end{equation*}
and it is the unique (weak) solution of the problem
\begin{equation}\label{W_ku_pb}
	\begin{cases}
		(-\Delta)^mW_{K,u}=0\quad\ \mbox{in}\;\Omega\setminus K,\\
		W_{K,u}\in V^m(\Omega),\\
		W_{K,u}-u\in V^m_0(\Omega\setminus K),
	\end{cases}
\end{equation}
in the sense that $W_{K,u}\in V^m(\Omega)$, $W_K-u\in V^m_0(\Omega\setminus K)$ and
\begin{equation*}
	\int_{\Omega\setminus
          K}\nabla^mW_{K,u}\nabla^m\varphi=0\qquad\mbox{for all}\;
        \varphi\in V^m_0(\Omega\setminus K).
\end{equation*}
In \eqref{W_ku_pb} we are in fact prescribing homogeneous Dirichlet or
Navier boundary conditions on $\dOmega$ and, in case $\partial K$ is
smooth, an ``$m$-Dirichlet-matching'' between $W_{K,u}$ and $u$ on
$\dK$, i.e. the $m$ conditions $W_{K,u}=u$,
$\partial_nW_{K,u}=\partial_nu$, $\dots$,
$\partial_n^{m-1}W_{K,u}=\partial_n^{m-1}u$ on $\dK$.

In particular, the minimizer $W_K$ of the $V^m$-capacity is such that
\begin{equation*}
	\capVm(K)=\int_\Omega|\nabla^mW_K|^2
\end{equation*}
and it is the unique (weak) solution of the problem
\begin{equation*}
	\begin{cases}
		(-\Delta)^mW_K=0\quad\mbox{in}\;\Omega\setminus K,\\
		W_K\in V^m(\Omega),\\
		W_K-\eta_K\in V^m_0(\Omega\setminus K).
	\end{cases}
\end{equation*}
in the sense that $W_K\in V^m(\Omega)$, $W_K-\eta_K\in V^m_0(\Omega\setminus K)$ and
\begin{equation}\label{W_k_sol}
  \int_{\Omega\setminus K}\nabla^mW_K\nabla^m\varphi=0
  \qquad\mbox{for all}\;\varphi\in  V^m_0(\Omega\setminus K).
      \end{equation}
We observe that $\capVm(K)=0$ implies that 
      $0\in S_{\eta_K}$, i.e. $\eta_K\in V^m_0(\Omega\setminus
      K)$. Since $\eta_K\equiv1$ on $K$, this can only hold true when
      the Sobolev space ``does not see'' $K$, i.e. when
      $V^m_0(\Omega\setminus K)=V^m(\Omega)$. As a consequence, the
      eigenvalues of problems \eqref{eq} and
        \eqref{eq_Ke_Nav} coincide with those of \eqref{eq_unperturbed}. 
      More precisely, in the spirit of \cite[Propositions 2.1 and
      2.2]{Courtois} (see also \cite[Proposition 3.3]{FNO}), we
      prove the following.
\begin{prop}\label{Capacity_zero}
	The following statements are equivalent:
	\begin{enumerate}[i)]
		\item $\capVm(K)=0$;
		\item $V^m_0(\Omega\setminus K)=V^m(\Omega)$;
		\item $\lambda_n(\Omega\setminus K)=\lambda_n(\Omega)\,$ for all $n\in\N$.
	\end{enumerate}
\end{prop}
\begin{proof}
  To show $(i)\Rightarrow(ii)$, by density of $X^m(\Omega)$ in
  $V^m(\Omega)$, see \eqref{Xm}, it is enough to prove that each
  $u\in X^m(\Omega)$ may be approximated by functions in
  $V^m_0(\Omega\setminus K)$ in the $V^m$-norm. Since $\capVm(K)=0$,
  there exists $(w_i)_i\subset V^m(\Omega)$ with
  $w_i-\eta_K\in V^m_0(\Omega\setminus K)$ so that
  $\|\nabla^mw_i\|_2^2\to0$ as $i\to+\infty$. Hence, defining
  $v_i:=u(1-\eta_Kw_i)$, one has that
  $v_i\in V^m_0(\Omega\setminus K)$ and, in view of
  \eqref{eq:equiv-norme},
	\begin{equation*}
		\begin{split}
                  \|\nabla^m(u-v_i)\|_2^2&=\|\nabla^m(u\eta_Kw_i)\|_2^2
                  \les\sum_{j=0}^m\intOmega|D^{m-j}(\eta_Ku)|^2|D^{j}w_i|^2\\
                  &\leq\|\eta_Ku\|_{W^{m,\infty}(\Omega)}^2\sum_{j=0}^m\intOmega|D^j w_i|^2=\|\eta_Ku\|_{W^{m,\infty}(\Omega)}^2\|w_i\|^2_{H^m(\Omega)}\\
                  &\leq C^2\|\eta_Ku\|_{W^{m,\infty}(\Omega)}^2
                  \|\nabla^m w_i\|_2^2\to 0
		\end{split}
	\end{equation*}
	as $i\to+\infty$.
	
	The reversed implication $(ii)\Rightarrow(i)$ is due to the fact that $\varphi=W_K$ may be used as a test function in \eqref{W_k_sol} to obtain that $\|W_K\|_{V^m_0(\Omega\setminus K)}=\|W_K\|_{V^m(\Omega)}=0$, which is equivalent to $(i)$.
	
	$(ii)\Rightarrow(iii)$ easily follows from the minimax
        characterization of the eigenvalues \eqref{var_char_eigv}. The
        converse is implied by the spectral theorem, because by
        $(iii)$ one is able to find an orthonormal
        basis of $V^m(\Omega)$ made of
        $V^m_0(\Omega\setminus K)$-functions.
\end{proof}
\begin{remark}\label{equiv_cap_0}
From Proposition \ref{Capacity_zero}, in particular
  from the implication $(i)\Rightarrow(ii)$, one derives the following equivalence:
$$\capVm(K)=0\quad\;\Leftrightarrow\quad\;\capVm(K,u)=0\;\,\mbox{for all}\;\, u\in V^m(\Omega).$$
\end{remark}

Next, we investigate some properties of the above defined capacities, in particular the monotonicity properties with respect to $\Omega$ and $K$, and the relation between the Dirichlet and the Navier capacities.
\begin{prop}[Monotonicity properties of the capacity]\label{Capacity_monotonicity}
	The following properties hold.
	\begin{enumerate}[i)]
		\item If $K_1\subset K_2\subset\Omega$,
                  $K_1,K_2$ are compact, and $h\in
                  V^m(\Omega)$, then
                  \begin{equation*}
                    \capVm(K_1,h)\leq\capVm(K_2,h).
                  \end{equation*}
		\item If $K\subset\Omega_1\subset\Omega_2$, $K$ is
                  compact,  and $h\in H^m(\Omega_2)$, then $\capp_{m,\Omega_2}(K,h)\leq\capp_{m,\Omega_1}(K,h)$.
		\item For every $K\subset\Omega$ compact and $h\in
                  H^m(\Omega)$,
 there holds $\capmND(K,h)\leq\capm(K,h)$.
	\end{enumerate}
\end{prop}
\begin{proof}
	\textit{i}) It is enough to notice that, for $u\in V^m(\Omega)$, the condition $u-h\in V^m_0(\Omega\setminus K_2)$ is more restrictive than $u-h\in V^m_0(\Omega\setminus K_1)$.
	
	\textit{ii}) Any $u\in H^m_0(\Omega_1)$ can be extended by $0$
        to a function in $H^m_0(\Omega_2)$, so the minimization for
        $\capp_{m,\Omega_2}(K,h)$ takes into consideration a larger
        set of test functions than the one for $\capp_{m,\Omega_1}(K,h)$,
        and consequently the $\inf$ decreases.
	
	\textit{iii}) It follows directly from the inclusions in \eqref{inclusion}.
\end{proof}

\begin{remark}
	Note that the argument used in the proof of
        (\textit{ii}) for  Dirichlet BCs is no more available in the
        case of Navier BCs on $\dOmega$.
\end{remark}

As an example, which is also relevant for our purposes, we compute the capacity of a point in~$\R^N$.
\begin{prop}[Capacity of a point]\label{Capacity_point}
  Let $x_0\in\Omega$. Then $\capVm(\{x_0\})=0$ if $N\geq2m$, while
  $\capVm(\{x_0\})>0$ when $N\leq 2m-1$.
\end{prop}
\begin{proof}
 It is not restrictive to assume that
  $x_0=0\in\Omega$. If $N\leq2m-1$, then the embedding
  $V^m(\Omega)\hookrightarrow C^0(\Omegabar)$ is
  continuous i.e. $\|\nabla^m u\|_2\geq C(m,N,\Omega)\|u\|_\infty$ for all
  $u\in V^m(\Omega)$, with a constant $C(m,N,\Omega)>0$ which does not depend on $u$. In
  particular for those functions in $V^m(\Omega)$ for which $u(0)=1$,
  one has $\|\nabla^m u\|_2\geq C(m,N,\Omega)$. Hence the infimum in the
  definition of $\capVm(K)$ is strictly positive.
	
  In view of Proposition
  \ref{Capacity_monotonicity}(iii), it is sufficient to prove it for
  the Dirichlet case. Let $N\geq 2m+1$ and take a sequence of shrinking
  cut-off in the following way: let $\zeta\in C^\infty_0(B_2(0))$ such
  that $\zeta\equiv1$ on $B_1(0)$ and consider
  $\zeta_k(x):=\zeta(kx)$. One has that
  $\zeta_k\in C^\infty_0(B_{\frac2k}(0))$ and $\zeta_k\equiv1$ on
  $B_{\frac1k}(0)$, hence $\supp\,\zeta_k\subset\Omega$
  for $k\geq k_0=k_0(\dist(0,\dOmega))$. We compute
  \begin{equation*}
    \int_\Omega|\nabla^m\zeta_k|^2=\int_{B_{\frac2k}(0)}|k^m(\nabla^m\zeta)(kx)|^2dx=k^{2m-N}\int_{B_2(0)}|\nabla^m\zeta|^2\to0
  \end{equation*}
  as $k\to\infty$ since $2m-N<0$. Being such functions
  admissible for the minimization of $\capm$, we deduce
  $\capm(\{0\})=0$. The argument is similar for the case $N=2m$,
  provided we choose accurately the sequence of cut-off functions, see
  \cite[Proposition 7.6.1/2 and Proposition 13.1.2/2]{M}. For the sake
  of completeness, we retrace here the proof. Let $\alpha$ denote a
  function in $C^\infty\left([0,1]\right)$ equal to zero near $t=0$,
  to $1$ near $t=1$, and such that $0\leq\alpha(t)\leq1$. Define then
  $\zeta_\varepsilon:=\alpha(v_\varepsilon)$, where \begin{equation*}
    v_\varepsilon(x):=\begin{cases}
      1&\text{if }|x|\leq\varepsilon,\\
      \frac{\log|x|-\log\sqrt\varepsilon}{\log\varepsilon-\log\sqrt\varepsilon}&\text{if }\varepsilon\leq|x|\leq\sqrt\varepsilon,\\
      0&\text{if }|x|\geq\sqrt\varepsilon. \end{cases} \end{equation*}
  Notice that $v_\varepsilon$ is continuous but not $C^1$; on the
  other hand 
  $\zeta_\varepsilon\in C^\infty_0(B_{\sqrt\varepsilon}(0))$, since $\alpha$ is constant in a
  neighbourhood of $0$ and in a neighbourhood of $1$ by
  construction. Therefore 
  $\zeta_\varepsilon\in H^m_0(B_1(0))$ for any
  $\varepsilon\in(0,1)$. Moreover  $\zeta_\varepsilon\equiv 1$ in
  $B_{\varepsilon}(0)$ so that $\zeta_\varepsilon$ is an admissible
  test function in the minimization of $\capm(\{0\})$. 
By direct calculations, there exists a constant $C=C(m)>0$
(independent of $\varepsilon$)
such that
\begin{equation*}
  |\nabla^m\zeta_\varepsilon(x)|\leq
  \frac{C}{|\log\varepsilon|}\frac1{|x|^m}\quad\text{for all }\varepsilon<|x|<\sqrt\varepsilon,
\end{equation*}
whereas
\begin{equation*}
  \nabla^m\zeta_\varepsilon(x)=0\quad\text{if either $|x|\leq
    \varepsilon$ or $|x|\geq \sqrt\varepsilon$}.
\end{equation*}
Therefore 
\begin{equation*}
		\int_\Omega|\nabla^m\zeta_\varepsilon|^2\les\frac1{\log^2\varepsilon}\int_\varepsilon^{\sqrt\varepsilon}\frac1r\,dr=\frac1{2|\log\varepsilon|}\to0
	\end{equation*}
	as $\varepsilon\to0$. The argument is concluded as above. 
      \end{proof}

\subsection{Homogeneous Sobolev
  spaces and capacities in $\boldsymbol{\R^N}$}\label{Sec_Cap_RN}

\subsubsection{The homogeneous Sobolev spaces $\boldsymbol{D^{m,2}_0(\R^N)}$}\label{sec:homog-sobol-spac}
So far, we defined the notion of $V^m$-capacity for compact sets
contained in an open bounded smooth domain $\Omega\subset\R^N$. An analogous definition can
be given when $\Omega=\R^N$, provided the underlined space is of
homogeneous kind. We introduce the homogeneous Sobolev spaces
(sometimes referred to as Beppo Levi spaces)
$D^{m,2}_0(\R^N)$ as the completion of
  $C^\infty_0(\R^N)$ with respect to the norm
\begin{equation*}
\|u\|_{D^{m,2}_0(\R^N)}:=\left(\int_{\R^N}|\nabla^mu|^2\right)^\frac12.
\end{equation*}
Actually, the spaces $D^{m,2}_0(\R^N)$ are more commonly
  defined as the completion with respect to the norm $\|D^m\cdot\|_2$,
  i.e. with respect to the full tensor of all highest
  derivatives. However, the two definitions are
    equivalent since, by integration by parts, $\|D^m\cdot\|_2$ and $\|\nabla^m\cdot\|_2$ are
    equivalent norms on $C^\infty_0(\R^N)$, see e.g. \cite[Sec.2.2.1]{GGS}.

For large dimensions $N>2m$, the following Sobolev inequalities are well-known: for
  every $0\leq j\leq m$ there
  exists a positive constant $S(N,m,j)>0$ (depending only on $N$, $m$
  and $j$) such that
  \begin{equation}\label{eq:sob-ineq}
    S(N,m,j)\left(\int_{\R^N}|D^ju|^{2^*_{m,j}}\right)^{\frac{2}{2^*_{m,j}}}\leq
    \|D^mu\|^2_{L^2(\R^N)}\quad\text{for all }u\in C^\infty_0(\R^N),   
  \end{equation}
where $2^*_{m,j}:=\frac{2N}{N-2(m-j)}$. In particular, for $j=0$,
there exists a positive constant $S(N,m)>0$ such that
  \begin{equation*}
    S(N,m)\left(\int_{\R^N}|u|^{2^*_m}\right)^{\frac{2}{2^*_m}}\leq
    \|u\|_{D^{m,2}_0(\R^N)}\quad\text{for all }u\in C^\infty_0(\R^N),
  \end{equation*}
  where $2^*_m:=2^*_{m,0}=\frac{2N}{N-2m}$, see \cite[Theorem 2.3]{GGS}.
In view of \eqref{eq:sob-ineq}, if $N>2m$, one may also characterize 
$D^{m,2}_0(\R^N)$ as
\begin{equation*}
  D^{m,2}_0(\R^N)=\big\{u\in L^{2^*_m}(\R^N)\,\big|\,
D^j u\in L^{2^*_{m,j}}(\R^N)\text{ for all }0<j\leq m\big\}.
\end{equation*}
Analogously, for $K\subset\R^N$ compact, one may consider the exterior
domain $\Omega=\R^N\setminus K$ and define $D^{m,2}_0(\R^N\setminus
K)$  as the completion of $C^\infty_0(\R^N\setminus
K)$ with respect to the norm $\|\nabla^m\cdot\|_2$, which is characterized, for $N>2m$, as
$$
D^{m,2}_0(\R^N\setminus K)=\bigg\{
\begin{array}{ll}
u\in L^{2^*_m}(\R^N\setminus
K)\,\big|\,&
D^j u\in L^{2^*_{m,j}}(\R^N\setminus K)\text{ for all }0<j\leq m
                    \\[3pt]&\mbox{and}\ \psi u\in
                              H^m_0(\R^N\setminus K)\ \mbox{for all}\
\psi\in C^\infty_0(\R^N)
\end{array}
\bigg\}\,,$$ see \cite[Theorem II.7.6]{Galdi}.

\subsubsection{A Hardy-Rellich-type inequality with intermediate derivatives}\label{Sec_Hardy}
Besides Sobolev inequalities, an important tool in the theory of
Sobolev spaces in large dimensions $N>2m$ is represented by
Hardy-Rellich inequalities, which state that the Sobolev norm of the
highest order derivatives controls a singularly weighted Sobolev norm
of the function. We refer to \cite{DH} for such inequalities in
$H^m_0(\Omega)$ and to \cite{GGM,GGS_art} for their extensions to
$H^m_\vartheta(\Omega)$. In this section, inspired by \cite{PlP}, we
prove a Hardy-Rellich-type inequality for the space
$H^2_\vartheta(\Omega)$ including also the gradient term, which
provides a further characterization of the space $D^{2,2}_0(\R^N)$ for
$N>4$. It will be needed in Section \ref{Sec_blowup} to identify the
functional space containing the limiting profile in the blow-up
argument, when Navier BCs are imposed on $\dOmega$.
\begin{thm}\label{HR_ineq}
  Let $N>4$ and $\Omega\subset\R^N$ be a smooth bounded domain. Then,
  for every function $u\in H^2(\Omega)\cap H^1_0(\Omega)$, one has that
  $\frac u{|x|^2},\,\frac{\nabla u}{|x|}\in L^2(\Omega)$ and
	\begin{equation}\label{eq:hardy-inter_m=2}
		(N-4)^2\intOmega\frac{|u|^2}{|x|^4}\,dx+2(N-4)\intOmega\frac{|\nabla u|^2}{|x|^2}\,dx\leq\intOmega|\Delta u|^2\,dx.
	\end{equation}
\end{thm}
\begin{proof}
  Let $u\in C^\infty(\Omegabar)$ be such that $u|_{\dOmega}=0$. Let us
  assume that $0\in\Omega$. Let us
  introduce a parameter $\lambda$ to be fixed later and, for
  $\varepsilon>0$ small, let us denote
  $\Omega_\varepsilon:=\Omega\setminus B_\varepsilon(0)$. We have that 
	\begin{equation}\label{HR_eq1}
		\begin{split}
			0&\leq\int_{\Omega_\varepsilon}\left(\frac x{|x|}\Delta u+\lambda u\frac x{|x|^3}\right)^2dx=\int_{\Omega_\varepsilon}(\Delta u)^2+\lambda^2\int_{\Omega_\varepsilon}\frac{u^2}{|x|^4}\,dx+2\lambda\int_{\Omega_\varepsilon}\frac u{|x|^2}\Delta u\,dx.
		\end{split}
	\end{equation}
	We can rewrite the third term as
	\begin{equation*}
		\begin{split}
			\int_{\Omega_\varepsilon}\frac u{|x|^2}\Delta u\,dx&=-\int_{\Omega_\varepsilon}\nabla u\left(\frac{\nabla u}{|x|^2}-2u\frac x{|x|^4}\right)dx+\int_\dOmega\frac u{|x|^2}\partial_\nu u\,d\sigma-\frac1{\varepsilon^2}\int_{\partial B_\varepsilon}u\nabla u\cdot\frac x\varepsilon\,d\sigma\\
			&=-\int_{\Omega_\varepsilon}\frac{|\nabla u|^2}{|x|^2}\,dx+\int_{\Omega_\varepsilon}\nabla(u^2)\frac x{|x|^4}\,dx+\Oo(\varepsilon^{N-3})\\
			&=-\int_{\Omega_\varepsilon}\frac{|\nabla u|^2}{|x|^2}\,dx-(N-4)\int_{\Omega_\varepsilon}\frac{u^2}{|x|^4}\,dx+\int_\dOmega u^2\frac{x\cdot\nu}{|x|^4}\\
			&\quad-\int_{\partial B_\varepsilon}\frac{u^2}{\varepsilon^3}\,d\sigma+\Oo(\varepsilon^{N-3})
		\end{split}
	\end{equation*}
as $\varepsilon \to0$. Since the third term vanishes and the second to last term is $\Oo(\varepsilon^{N-4})$ as $\varepsilon\to0$, from \eqref{HR_eq1} we get
\begin{equation*}
	0\leq\int_{\Omega_\varepsilon}(\Delta u)^2+\lambda^2\int_{\Omega_\varepsilon}\frac{u^2}{|x|^4}\,dx-2\lambda\int_{\Omega_\varepsilon}\frac{|\nabla u|^2}{|x|^2}\,dx-2\lambda(N-4)\int_{\Omega_\varepsilon}\frac{u^2}{|x|^4}\,dx+\Oo(\varepsilon^{N-4}).
\end{equation*}
Choosing now $\lambda=N-4$, we obtain
\begin{equation*}
	(N-4)^2\int_{\Omega_\varepsilon}\frac{u^2}{|x|^4}\,dx+2(N-4)\int_{\Omega_\varepsilon}\frac{|\nabla
          u|^2}{|x|^2}\,dx\leq\int_{\Omega_\varepsilon}(\Delta
        u)^2+\Oo(\varepsilon^{N-4}) \quad\text{as }\varepsilon\to0.
\end{equation*}
Inequality \eqref{eq:hardy-inter_m=2} follows by letting
$\varepsilon\to0$ and by density of the set  $\{u\in
C^\infty(\Omegabar)\,|\,u|_{\dOmega}=0\}$ in $H^2(\Omega)\cap
H^1_0(\Omega)$. If $0\not\in\Omega$ the above argument can be
  repeated by considering directly  in \eqref{HR_eq1} the integral on
  the whole $\Omega$.
\end{proof}

We observe that  \eqref{eq:hardy-inter_m=2} holds also for all functions in $C^\infty_0(\R^N)$ (since any function $u\in C^\infty_0(\R^N)$ is contained in some $H^2_\vartheta(\Omega)$). Therefore, by density of $C^\infty_0(\R^N)$ in $D^{2,2}_0(\R^N)$ and Fatou's Lemma we easily deduce that, if $N>4$, then 
\begin{equation*}
    (N-4)^2\int_{\R^N}\frac{|u|^2}{|x|^4}\,dx+2(N-4)\int_{\R^N}\frac{|\nabla u|^2}{|x|^2}\,dx\leq\int_{\R^N}|\Delta u|^2\,dx
  \end{equation*}
  for all $u\in D^{2,2}_0(\R^N)$.
In particular we have that $D^{2,2}_0(\R^N)$ is contained in the space
$$\cS^2(\R^N):=\left\{u\in
  H^2_{loc}(\R^N)\,\Big|\,\frac{\nabla^{2-k}u}{|x|^k}\in
  L^2(\R^N)\text{ for }k\in\{0,1,2\}\right\}.$$
We prove now that the two functional spaces coincide.
\begin{prop}\label{HR_spaces}
	$\cS^2(\R^N)=D^{2,2}_0(\R^N)$ for all $N>4$.
\end{prop}
\begin{proof}
	We have already observed above that
        $\cS^2(\R^N)\supseteq D^{2,2}_0(\R^N)$. Let now
        $u\in\cS^2(\R^N)$, $\eta$ be a cutoff function with support in
        $B_2(0)$ and which takes the value $1$ in $B_1(0)$, and define
        $\eta_R:=\eta\left(\tfrac\cdot R\right)$ for all $R>0$. Hence
        $\eta_Ru\in H^2_0(B_{2R}(0))$ and we claim that
        $\|\Delta(\eta_Ru- u)\|_{2}\to0$ as $R\to+\infty$. Indeed,
	\begin{equation*}
		\|\Delta\big(\left(\eta_R-1\right)u\big)\|_2^2\les\|(\Delta\eta_R)u\|_2^2+\|\nabla\eta_R\nabla u\|_2^2+\|\left(\eta_R-1\right)\Delta u\|_2^2,
	\end{equation*}
	where
	\begin{equation*}
		\|\left(\eta_R-1\right)\Delta u\|_2^2\leq\int_{\R^N\setminus B_R}|\Delta u|^2\to 0
	\end{equation*}
	as $R\to+\infty$, and for $k\in\{1,2\}$,
	\begin{equation*}\label{S2D22_2}
		\begin{split}
			\|\nabla^k\eta_R\nabla^{2-k}u\|_2^2&=\int_{R<|x|<2R}\frac1{R^{2k}}\left|\big(\nabla^k\eta\big)\left(\frac xR\right)\right|^2|\nabla^{2-k}u|^2\,dx\\
			&\les 2^{2k}\int_{\R^N\setminus B_R(0)}\frac{|\nabla^{2-k}u|^2}{|x|^{2k}}\,dx\to0
		\end{split}
	\end{equation*}
	as $R\to+\infty$. By density of $C^\infty_0(B_{2R}(0))$
          in $H^2_0(B_{2R}(0))$, this implies that $C^\infty_0(\R^N)$
        is dense in $\cS^2(\R^N)$ in the $D^{2,2}_0$-norm, thus
        concluding the proof.
\end{proof}

\subsubsection{Capacities in $\boldsymbol{\R^N}$}
Similarly to the case of a bounded set $\Omega$ described in  Section \ref{Sec_cap_Omega}, for any compact set $K\subset\R^N$ and any $u\in
D^{m,2}_0(\R^N)$ with $N>2m$, we define
\begin{equation}\label{capp_RN}
	\capp_{m,\R^N}(K,u):=\inf\left\{\int_{\R^N}|\nabla^mf|^2\,\Big|\,f\in D^{m,2}_0(\R^N),\, f-u\in D^{m,2}_0(\R^N\setminus K)\right\},
\end{equation}
which we simply denote by $\capp_{m,\R^N}(K)$ when $u=\eta_K$. The
argument for the attainability of the capacity is easily adapted from
the one for $\capp_{V^m,\Omega}(K,u)$. Analogous properties hold also
in this setting, in particular it is true that
\begin{equation}\label{eq:dmek}
  D^{m,2}_0(\R^N)=D^{m,2}_0(\R^N\setminus
  K)\quad\text{if and only if}\quad\capp_{m,\R^N}(K)=0
\end{equation}
which directly implies that
$$\capp_{m,\R^N}(K)=0\quad\;\Leftrightarrow\quad\;\capp_{m,\R^N}(K,u)=0\;\,\mbox{for all}\;\, u\in D^{m,2}_0(\R^N).$$
The analogue of \eqref{eq:dmek} in the
case of a bounded domain $\Omega$ is contained in Proposition
\ref{Capacity_zero} and its proof relies on \eqref{eq:equiv-norme},
which in turn is based on a
Poincar\'{e} inequality, the latter being no longer
valid in $\R^N$. However, if $N>2m$, the role played
  by Poincar\'{e} inequalities can be replaced by the critical Sobolev embedding. Although
known, here we retrace the proof of \eqref{eq:dmek} for the sake of completeness. Let
$u\in C^\infty_0(\R^N)$, set $\Sigma:=\supp(u)$, and consider
$(w_i)_i\subset D^{m,2}_0(\R^N)$ with
$w_i-\eta_K\in D^{m,2}_0(\R^N\setminus K)$ such that
$\|\nabla^mw_i\|_2^2\to0$ as $i\to+\infty$. Then
$v_i:=u(1-w_i)\in D^{m,2}_0(\R^N\setminus K)$ and, defining
$q_j:=2^*_{m,j}=\frac{2N}{N-2(m-j)}\geq2$ for
$j\in\{0,\dots,m\}$, one has that
\begin{equation*}
	\begin{split}
		\|\nabla^m(u-v_i)\|_2^2&=\|\nabla^m(uw_i)\|_{L^2(\Sigma)}^2\les\|u\|_{W^{m,\infty}(\R^N)}^2\sum_{j=0}^m\int_{\Sigma}|D^jw_i|^2\\
		&\les\sum_{j=0}^m\left(\int_{\Sigma}|D^jw_i|^{q_j}\right)^{\frac2{q_j}}\leq\sum_{j=0}^m\|D^jw_i\|_{L^{q_j}(\R^N)}^2\\
		&\les\|D^mw_i\|_{L^2(\R^N)}^2\les\|\nabla^mw_i\|_{L^2(\R^N)}^2\to0,
	\end{split}
\end{equation*}
where in the last steps we used H\"older inequality, the
Sobolev inequality \eqref{eq:sob-ineq}, and the equivalence of the norms $\|D^m\cdot\|_2$ and $\|\nabla^m\cdot\|_2$.

\vskip0.2truecm
For later use, we also recall the right continuity of the capacity, see \cite[Sec. 13.1.1]{M}. 
\begin{lem}\label{capacity_right_continuity}
		Let $K$ be a compact subset of
                $\Omega\subset\R^N$. For any $\varepsilon>0$ there
                exists a neighbourhood $\cU(K)\subset\Omega$ such that
                for any compact set $\widetilde K$ with
                $K\subset\widetilde K\subset\cU(K)$, there
                holds
		\begin{equation*}
			\capp_{m,\Omega}(\widetilde K)\leq\capp_{m,\Omega}(K)+\varepsilon.
		\end{equation*}
\end{lem}

Although the notion of capacity needed for the blow-up analysis in
Section \ref{Sec_blowup} is the one given in \eqref{capp_RN}, sometimes it is
useful to consider a second one defined as 
\begin{equation}\label{Capp}
	\Capp(K):=\inf\left\{\int_{\R^N}|\nabla^mf|^2\,\Big|\,f\in D^{m,2}_0(\R^N),\, f\geq1\;\mbox{a.e. on}\;K\right\},
\end{equation}
which is well-defined for $N>2m$, and similarly ${\rm Cap}_{m,\Omega}^\geq$ for $\Omega\subset\subset\R^N$, see \cite{M,M1}. One of the advantages in this approach is that the capacitary potential associated to $\Capp$ is positive, see \cite[Sec. 3.1.2]{GGS}. 
Note that for all $\Omega\subset\R^N$ one has ${\rm
  Cap}_{m,\Omega}^\geq(K)\leq\capp_{m,\Omega}(K)$ because the class of
test functions considered in \eqref{Capp} includes the one considered
for the minimization in \eqref{capp_RN}. Actually it turns out that
the two definitions are equivalent, in the sense that
  the two capacities estimate each other, as stated below. We report here the result for $\Omega=\R^N$, referring to \cite{M1} for the general case $\Omega\subsetneq\R^N$.
\begin{lem}(\cite{M}, Theorem 13.3.1)\label{Lemma_Mazya_equivalence}
	Let $m\in\N\setminus\{0\}$ and $N>2m$. There exists a constant $c>0$ such that
	\begin{equation*}
		c\,\capp_{m,\R^N}(K)\leq\Capp(K)\leq \capp_{m,\R^N}(K)
	\end{equation*}
	for any compact set $K\subset\R^N$.
\end{lem}
\begin{remark}
  The constant
  $c$ appearing in Lemma \ref{Lemma_Mazya_equivalence} can be taken
  $1$ in the second-order case
  $m=1$, so the two definitions coincide, see
  e.g. \cite[Sec. 13.3]{M}. Whether this is the case also for the
  higher-order case it is still an open question.
\end{remark}
\begin{remark}\label{Remark_manifolds}
  As an extension of Proposition \ref{Capacity_point}, it is known
  that a regular manifold of dimension
  $d$ has zero capacity in the sense of \eqref{Capp} if and only if
  $d\leq
  N-2m$, see \cite[Corollary 5.1.15]{AH}. By Lemma
  \ref{Lemma_Mazya_equivalence}, this result holds also for the notion
  \eqref{capp_RN} of capacity.
\end{remark}

\section{Convergence and asymptotic expansion of the perturbed eigenvalues}\label{Sec_conv_eigv}
In this section we study stability and asymptotic expansion of the
perturbed eigenvalues of \eqref{eq} and \eqref{eq_Ke_Nav}, when from a
bounded domain $\Omega\subset\R^N$ one removes a compact set $K$ of
small $V^m$-capacity. The main goal is to extend the results obtained
in the second-order framework (in particular \cite[Theorem 1.4]{AFHL})
to the higher-order settings described in the introduction. The first
part is devoted to the proof of the stability result of Theorem
\ref{Alternative_conv_eigv}, which applies for rather general domains,
while in the second part we focus on the asymptotic expansion of
simple eigenvalues contained in Theorem \ref{Prop_eigv_cap}, for which
we require the notion of concentrating family of compact sets.

\subsection{Spectral stability: Proof of Theorem \ref{Alternative_conv_eigv}}\label{Section_spectral_stability}
We present here a simple and self-contained proof of the stability of
the point spectrum of the polyharmonic operator with respect to the
capacity of the removed compactum, in both Dirichlet and Navier
settings described in Section \ref{Sec_spaces}. It is essentially based on
the variational characterization of the eigenvalues
\eqref{var_char_eigv} and on the properties of the capacitary
potentials, and it follows some ideas exploited for the same question
in the second-order case in \cite[Theorem 1.2]{AFN}.

\begin{proof}[Proof of Theorem \ref{Alternative_conv_eigv}]
  Denote by $(u_i)_{i=1}^\infty$ an orthonormal basis of $L^2(\Omega)$
  such that each $u_i$ is an eigenfunction of problem
  \eqref{eq_unperturbed} associated to the eigenvalue
  $\lambda_i(\Omega)$. By classical elliptic regularity
  theory (see e.g. \cite[Section 2.5]{GGS}), the
  smoothness of $\dOmega$ yields $u_i\in C^m(\Omegabar)$ for all
  $i\in\N$. In order to deal at once with both cases
    (\textit{D}) and (\textit{N}), we introduce the function $H$
    defined by $H\equiv1$ in the Dirichlet case, and by $H=\eta_{K_0}$
    in the Navier case. Here $\eta_{K_0}$ is a cutoff function which
    is equal to $1$ in a neighbourhood on $K_0$ and with support
    contained in  some compact set $\widetilde{K_0}$ such
    that 
    $K_0\subset\widetilde{K_0}\subset\Omega$. The cutoff
    $\eta_{K_0}$ is introduced in order to enforce the boundary
    conditions on $\dOmega$ in the Navier case.
    
  Fix $j\in\N\setminus\{0\}$. For any $\ell\in\{1,\dots, j\}$, we define
  $\Phi_\ell:=u_\ell(1-HW_K)$ and introduce the subspace
  $X_j:=\spann\{\Phi_\ell\}_{\ell=1}^j$. Note that
  $\Phi_\ell\in V^m_0(\Omega\setminus K)$ by definition of the
  capacitary potential $W_K$, so
  $X_j\subset V^m_0(\Omega\setminus K)$.  The aim is to prove that
  $X_j$ is a $j$-dimensional subspace of $V^m_0(\Omega\setminus K)$ so
  that the right hand side of \eqref{var_char_eigv} is smaller than
  the maximum of the Rayleigh quotient over $X_j$.  Note that, by
  trivially extending the functions
  $\left\{\Phi_\ell\right\}_{\ell=1}^j$ in $K$, the integrals may be
  evaluated on $\Omega$. First,
	\begin{equation*}
		\int_\Omega\Phi_h\Phi_\ell=\int_\Omega u_h u_\ell-2\int_\Omega u_h u_\ell\,H W_K+\intOmega u_h u_\ell\,H^2 W_K^2,
	\end{equation*}
	therefore, by orthonormality of $\{u_\ell\}_{\ell=1}^j$ in $L^2(\Omega)$ and \eqref{eq:equiv-norme},
	\begin{equation}\label{Phi_L2product}
		\begin{split}
			\left|\int_\Omega\Phi_h\Phi_\ell-\delta_{h,\ell}\right|&\leq\max_{1\leq h\leq j}\|u_h\|_{L^\infty(\Omega)}^2\left(2|\Omega|^{1/2}\|W_K\|_2+\|W_K\|_2^2\right)\\
			&\les\left(\capVm(K)\right)^{1/2}+\capVm(K),
		\end{split}
	\end{equation}
	where $\delta_{h,\ell}$ stands for the Kroenecker delta.
	Let now $(W_n)_n\subset X^m(\Omega)$, see \eqref{Xm}, be a sequence of smooth functions which approximates in the $V^m$-norm the capacitary potential $W_K$ and satisfying $W_n=1$ in $\cU(K)$. The existence of such a sequence is guaranteed by the definition of $W_K$. Define moreover $\Phi^\ell_n:=u_\ell(1-HW_n)$ for $n\in\N$. Note that $\Phi^\ell_n\to\Phi_\ell$ in $V^m(\Omega)$ as $n\to+\infty$. We get 
	\begin{equation}\label{Phi_Hmproduct}
		\begin{split}
			\intOmega\nabla^m\Phi^h_n\nabla^m\Phi^\ell_n&=\intOmega\nabla^m\left(u_h(1-HW_n)\right)\nabla^m\left(u_\ell(1-HW_n)\right)\\
			&=\intOmega\nabla^mu_h\nabla^mu_\ell(1-HW_n)^2+
			T_m(u_h,u_\ell,W_n),
		\end{split}
	\end{equation}
	where the term $T_m$ contains all remaining products between
        the derivatives of $u_h$, $u_\ell$, and $1-HW_n$. To deal with
        the first term on the right in \eqref{Phi_Hmproduct}, consider
        $u_\ell(1-HW_n)^2\in V^m(\Omega)$ by regularity of the
        factors, as a test function for the eigenvalue problem
        \eqref{weak_sol_eigfct} for $\lambda_h(\Omega)$. One obtains
	\begin{equation*}
		\begin{split}
			\lambda_h(\Omega)\intOmega\Phi^h_n\Phi^\ell_n&=\lambda_h(\Omega)\intOmega u_hu_\ell(1-HW_n)^2=\intOmega\nabla^mu_h\nabla^m\left(u_\ell(1-HW_n)^2\right)\\
			&=\intOmega\nabla^mu_h\nabla^mu_\ell(1-HW_n)^2+
			\intOmega\nabla^mu_h S_m(u_\ell,W_n),
		\end{split}
	\end{equation*}
	where again all remaining products involving intermediate derivatives of $u_\ell$ and $1-HW_n$ are collected in the term $S_m$ (which is a vector if $m$ is odd). Isolating the first term on the right hand-side, and substituting it into \eqref{Phi_Hmproduct}, we get
	\begin{equation}\label{Phi_Hmproduct_pt2}
		\intOmega\nabla^m\Phi^h_n\nabla^m\Phi^\ell_n-\lambda_h(\Omega)\intOmega\Phi^h_n\Phi^\ell_n=-\intOmega\nabla^mu_h S_m(u_\ell,W_n)+T_m(u_h,u_\ell,W_n).
	\end{equation}
	Moreover, 
	\begin{equation}\label{Phi_Hmproduct_pt3}
		\begin{split}
						\Big|\intOmega&\nabla^mu_hS_m(u_\ell,W_n)\Big|\\
            &\leq\sum_{i=1}^m\sum_{\tau=0}^i\intOmega|\nabla^mu_h||D^{m-i}u_\ell||D^{i-\tau}(1-HW_n)||D^\tau(1-HW_n)\Big|\\
			&\leq\|\nabla^mu_h\|_\infty\|u_\ell\|_{W^{m,\infty}(\Omega)}\sum_{i=1}^m\sum_{\tau=0}^i\|D^{i-\tau}(1-HW_n)\|_2\|D^\tau(1-HW_n)\|_2\\
			&\leq\max_{1\leq k\leq j}\|u_k\|_{W^{m,\infty}(\Omega)}^2\sum_{i=1}^m\!\bigg(2\|D^i(HW_n)\|_2\|1-HW_n\|_2\\
			&\hskip4.2cm+\sum_{\tau=1}^{i-1}\|D^{i-\tau}(HW_n)\|_2\|D^\tau (HW_n)\|_2\bigg)\\
			&\leq C(\Omega,j,m)\Big(\|H\|_{W^{m,\infty}(\Omega)}\|W_n\|_{H^m(\Omega)}(|\Omega|^{1/2}+\|W_n\|_2)\\
            &\hskip2.2cm+\|H\|_{W^{m,\infty}(\Omega)}^2\|W_n\|_{H^m(\Omega)}^2\Big)\\
			&\les
                          C(\Omega,j,m)\Big(\|H\|_{W^{m,\infty}(\Omega)}\left(\capVm(K)+o_n(1)\right)^{1/2}\\
&\hskip2.2cm+\|H\|_{W^{m,\infty}(\Omega)}^2(\capVm(K)+o_n(1))\Big)\quad\text{as }
n\to\infty
,
		\end{split}
	\end{equation}
	having used the equivalence of the norms
        $\|\cdot\|_{H^m(\Omega)}$ and $\|\nabla^m\cdot\|_2$ in
        $V^m(\Omega)$. Here $o_n(1)$ denotes a real sequence
          converging to $0$ as $n\to+\infty$.
Analogously one may estimate the last term in \eqref{Phi_Hmproduct_pt2}: 
	\begin{equation}\label{Phi_Hmproduct_pt4}
		\begin{split}
			|T_m&(u_h,u_\ell,W_n)|\leq\sum_{\substack{i,\tau\in\{0,\dots,m\}\\(i,\tau)\neq(0,0)}}\intOmega|D^{m-i}u_h||D^i(1-HW_n)||D^{m-\tau}u_\ell||D^\tau(1-HW_n)|\\
			&\leq\max_{1\leq h\leq j}\|u_h\|_{W^{m,\infty}(\Omega)}^2\bigg(2\sum_{\tau=1}^m\|D^\tau(HW_n)\|_2\|1-HW_n\|_2\\
			&\hskip3cm\qquad+\sum_{i,\tau\in\{1,\dots,m\}}\|D^i(HW_n)\|_2\|D^\tau(HW_n)\|_2\bigg)\\
			&\leq
                        C(\Omega,j,m)\Big(\|H\|_{W^{m,\infty}(\Omega)}\|W_n\|_{H^m(\Omega)}\left(|\Omega|^{1/2}+\|W_n\|_2\right)\\
&\hskip3cm\qquad 
+\|H\|_{W^{m,\infty}(\Omega)}^2\|W_n\|_{H^m(\Omega)}^2\Big)\\
	&\les
                          C(\Omega,j,m)\Big(\|H\|_{W^{m,\infty}(\Omega)}\left(\capVm(K)+o_n(1)\right)^{1/2}\\
&\hskip2.2cm+\|H\|_{W^{m,\infty}(\Omega)}^2(\capVm(K)+o_n(1))\Big)\quad\text{as }
n\to\infty.
		\end{split}
	\end{equation}
	All in all, from \eqref{Phi_Hmproduct_pt2}-\eqref{Phi_Hmproduct_pt4}, 
	one concludes
	\begin{multline*}
		\left|\intOmega\nabla^m\Phi^h_n\nabla^m\Phi^\ell_n-\lambda_h(\Omega)\intOmega\Phi^h_n\Phi^\ell_n\right|\\\leq\widetilde C\left(\left(\capVm(K)+o_n(1)\right)^{1/2}+\capVm(K)+o_n(1)\right),
	\end{multline*}
	where $\widetilde C$ depends on $K_0$ in the Navier case. Letting now $n\to+\infty$ in both sides of the inequality, and taking into account \eqref{Phi_L2product}, one infers
	\begin{equation}\label{Phi_Hmproduct_end}
		\left|\intOmega\nabla^m\Phi_h\nabla^m\Phi_\ell-\lambda_h(\Omega)\delta_{h,\ell}\right|\leq\widetilde C\left(\left(\capVm(K)\right)^{1/2}+\capVm(K)\right).
	\end{equation}
        Hence, from \eqref{Phi_L2product} and
        \eqref{Phi_Hmproduct_end} one sees that, when $\capVm(K)$ is
        small enough, the functions $\{\Phi_\ell\}_{\ell=1}^j$ are
        linearly independent in $V^m_0(\Omega\setminus K)$, and so the
        subspace $X_j$ has dimension $j$. Therefore, recalling that
        $\lambda_h(\Omega)\leq\lambda_j(\Omega)$ for all
        $h\in\{1,\dots,j\}$, again from \eqref{Phi_L2product} and
        \eqref{Phi_Hmproduct_end} one finally infers that
	\begin{equation*}
		\begin{split}
			\lambda_j(\Omega\setminus K)&\leq\max_{\substack{\left(\alpha_1,\dots,\alpha_j\right)\in\R^j\\\sum_{i=1}^j\alpha_i=1}}\frac{\sum\limits_{h,\ell=1}^j\alpha_h\alpha_\ell\intOmega\nabla^m\Phi_h\nabla^m\Phi_\ell}{\sum\limits_{h,\ell=1}^j\alpha_h\alpha_\ell\intOmega\Phi_h\Phi_\ell}\\
			&\leq\max_{\substack{\left(\alpha_1,\dots,\alpha_j\right)\in\R^j\\\sum_{i=1}^j\alpha_i=1}}\frac{\sum\limits_{h=1}^j\alpha_h^2\lambda_h(\Omega)+\Oo\left(\left(\capVm(K)\right)^{1/2}\right)}{\sum\limits_{h=1}^j\alpha_h^2+\Oo\left(\left(\capVm(K)\right)^{1/2}\right)}\\
			&\leq\frac{\lambda_j(\Omega)+\Oo\left(\left(\capVm(K)\right)^{1/2}\right)}{1+\Oo\left(\left(\capVm(K)\right)^{1/2}\right)}=\lambda_j(\Omega)+\Oo\left(\left(\capVm(K)\right)^{1/2}\right)
		\end{split}
	\end{equation*}
    as $\capVm(K)\to0$.
\end{proof}

\subsection{Asymptotic expansion of eigenvalues: Proof of Theorem \ref{Prop_eigv_cap}}\label{Section_asympt}

Let $\{K_\varepsilon\}_{\varepsilon>0}$ be a family of compact subsets
of $\Omega$ and denote by $\lambda_J(\Omega\setminus K_\varepsilon)$
the $J$-th eigenvalue of $(-\Delta)^m$ in
$V^m_0(\Omega\setminus K_\varepsilon)$, i.e. of problem
\eqref{weak_sol_eigfct} with $K=K_\varepsilon$.
If there exists a
limiting set $K$ for which $\capVm(K_\varepsilon)\to\capVm(K)=0$,
Theorem \ref{Alternative_conv_eigv} and Proposition
\ref{Capacity_zero} guarantee that
$\lambda_J(\Omega\setminus K_\varepsilon)\to\lambda_J(\Omega\setminus
K)=\lambda_J(\Omega)$,
if we denote by $\lambda_J(\Omega\setminus K)$ the corresponding
eigenvalue of the limiting problem in
$V^m_0(\Omega\setminus K)=V^m_0(\Omega)$. Moreover, Theorem
\ref{Alternative_conv_eigv} gives us a first estimate on the
eigenvalue convergence rate
in terms of the $V^m$-capacity of the removed set $K_\varepsilon$. Inspired by \cite{AFHL}, we are now going to sharpen this result, by detecting the first term of the asymptotic expansion of $\lambda_J(\Omega\setminus K_\varepsilon)$, provided the family of compact sets $\{K_\varepsilon\}_{\varepsilon>0}$ converges to $K$ as specified in Definition \ref{def_conv}. Indeed, as the next two propositions show, this definition of convergence, although very general, is enough to prove the stability of the $(u,V^m)$-capacity in case $\capVm(K)=0$, as well as the Mosco convergence of the functional spaces. 

\begin{prop}\label{cap_pot_stability}
  Let $\{K_\varepsilon\}_{\varepsilon>0}$ be a family of compact sets contained in $\Omega\subset\R^N$ concentrating to a compact set $K\subset\Omega$ with $\capVm(K)=0$ as $\varepsilon\to0$. Then, for every function $u\in V^m(\Omega)$, one has that $W_{K_\varepsilon,u}\to W_{K,u}=0$ strongly in $V^m(\Omega)$ and $\capVm(K_\varepsilon,u)\to\capVm(K,u)=0$ as $\varepsilon\to0$.
\end{prop}
\begin{proof}
	It is analogous to the one for the case $m=1$ given in \cite[Proposition B.1]{AFHL}. It is in fact essentially based on the fact that $V^m_0(\Omega\setminus K)=V^m(\Omega)$ for sets of null $V^m$-capacity, as shown in Proposition \ref{Capacity_zero}, and on the consequent Remark \ref{equiv_cap_0}.
\end{proof}

\begin{defn}
  Let $\{K_\varepsilon\}_{\varepsilon>0}$ be a family of compact sets
  compactly contained in a bounded domain $\Omega$. We say that
  $\Omega\setminus K_\varepsilon$ converges to $\Omega\setminus K$
  \textit{in the sense of Mosco} in $V^m_0$ if the following two
  conditions are satisfied:
	\begin{enumerate}
		\item[(\textit{i})] the weak limit points in $V^m(\Omega)$ of every family of functions $u_\varepsilon\in V^m_0(\Omega\setminus K_\varepsilon)$ belong to $V^m_0(\Omega\setminus K)$;
		\item[(\textit{ii})] for every $u\in V^m_0(\Omega\setminus K)$, there exists a family of functions $\{u_\varepsilon\}_{\varepsilon>0}$ such that, for every $\varepsilon>0$, $u_\varepsilon\in V^m_0(\Omega\setminus K_\varepsilon)$ and $u_\varepsilon\to u$ in $V^m(\Omega)$.
	\end{enumerate}
\end{defn}
In order to stress the underlined functional space, we also say that $V^m_0(\Omega\setminus K_\varepsilon)$ converges to $V^m_0(\Omega\setminus K)$ in the sense of Mosco.

\begin{lemma}\label{Mosco}
	Let $\{K_\varepsilon\}_{\varepsilon>0}$ be a family of compact sets concentrating to a compact set $K\subset\Omega$ with $\capVm(K)=0$ as $\varepsilon\to0$. Then $V^m_0(\Omega\setminus K_\varepsilon)$ converges to $V^m_0(\Omega\setminus K)$ as $\varepsilon\to0$ in the sense of Mosco.
\end{lemma}

\begin{proof}
  Verification of (\textit{i}). Let
  $\{u_\varepsilon\}_\varepsilon\subset V^m(\Omega)$ be such that
  $u_\varepsilon\in V^m_0(\Omega\setminus K_\varepsilon)$ and
  $u_\varepsilon\rightharpoonup u$ in $V^m(\Omega)$. Since
  $\capp_{m,\Omega}(K)=0$, we have that
    $V^m(\Omega)=V^m_0(\Omega\setminus K)$ by Proposition
    \ref{Capacity_zero}, hence $u$ belongs to
  $V^m_0(\Omega\setminus K)$.
	
  Verification of (\textit{ii}). Let
  $u\in V^m_0(\Omega\setminus K)=V^m(\Omega)$.
For every $k\in\N\setminus\{0\}$, by density  there exists 
  $\chi_k\in X^m_0(\Omega\setminus K)$ such that
  $\|\nabla^m(\chi_k-u)\|_2<\frac1k$. Note that, if
  $K_\varepsilon$ is concentrating to $K$ in the sense of Definition
  \ref{def_conv}, for a chosen cutoff function
  $\eta_K\in C^\infty_0(\Omega)$ such that $\eta_K\equiv1$ in a
  neighbourhood of $K$, one
  has that $\eta_K\equiv1$ in a
  neighbourhood of $K_\varepsilon$
  for $\varepsilon$
  small enough. 
  By definition of $W_K$, one may find $(W_n)_n\subset V^m(\Omega)$
  and a sequence $(\varepsilon_n)_n\searrow0$
  such that  $\|\nabla^mW_n\|_2<\frac1n$ and $W_n\equiv1$ in a
    neighbourhood of $K_\varepsilon$ for all $\varepsilon\in
    (0,\varepsilon_n]$. Defining,
  for all $n,k\in\N\setminus\{0\}$,
    $Z_n^k:=\chi_k\left(1-\eta_KW_n\right)$, one has that
  $Z_n^k\in V^m_0(\Omega\setminus K_\varepsilon)$ for all
  $\varepsilon\in(0,\varepsilon_n]$ and
	\begin{equation*}
			\|\nabla^m\left(Z_n^k-\chi_k\right)\|_2
			\les \|\eta_K\|_{W^{m,\infty}(\Omega)}\|W_n\|_{V^m(\Omega)}\|\chi_k\|_{W^{m,\infty}(\Omega)}\leq\frac {C_k}n
	\end{equation*}
for some $C_k>0$ depending on $k$.
        Hence, for each $k\in \N\setminus\{0\}$, there exists
        $n_k\in\N$ such that $n_k\nearrow\infty$ as $k\to\infty$ and
        $\|\nabla^m\left(Z_{n_k}^k-\chi_k\right)\|_2<\frac1k$. In
      order to construct the family required for the Mosco
      convergence, for any
      $\varepsilon\in(0,\varepsilon_{n_1})$ it is sufficient to
      define $u_\varepsilon:=Z_{n_k}^k$, choosing $k$ such that
      $\varepsilon\in(\varepsilon_{n_{k+1}},\varepsilon_{n_k}]$. Indeed,
      for any $\delta>0$, letting $k\in\N\setminus\{0\}$ be such
        that $\frac2k<\delta$, we have that, for all $\varepsilon\in
        (0,\varepsilon_{n_k}]$, $u_\varepsilon=Z_{n_j}^j$ for some
        $j\geq k$, so that
	$$\|\nabla^m\left(u_\varepsilon-u\right)\|_2\leq\|\nabla^m\big(Z_{n_j}^j
        -\chi_j\big)\|_2+\|\nabla^m\left(\chi_j-u\right)\|_2<\frac2j\leq\frac2k<\delta,$$
thus proving that $u_\varepsilon \to u$ in $V^m(\Omega)$ as $\varepsilon\to0$.      
\end{proof}

\begin{remark}
	Note that the Mosco convergence of sets implies the convergence of the spectra of the polyharmonic operators, see \cite{AL}. For the Dirichlet case, in particular this can be seen combining \cite[Proposition 2.9, footnote 2 p.8, and Theorem 4.3]{AL}.
\end{remark}

\begin{lemma}\label{Lemma_A1_adapt}
  Let $K\subset\Omega$ be a compact set and
  $\{K_\varepsilon\}_{\varepsilon>0}$ be a family of compact subsets
  of $\Omega$ concentrating to $K$ as $\varepsilon\to0$. If $\capVm(K)=0$, then, for every
  $f\in V^m(\Omega)$, we have that 
  $\|W_{K_\varepsilon,f}\|_{H^{m-1}(\Omega)}^2={\scriptstyle
    \mathcal{O}}(\capVm(K_\varepsilon,f))$ as $\varepsilon\to0$.
\end{lemma}
\begin{proof}
	The proof is inspired by \cite[Lemma A.1]{AFHL}. 
	Suppose by contradiction that there exist $C>0$ and a sequence $\varepsilon_n\to0$ such that
	\begin{equation}\label{Lemma_A1_adapt_contrad}
	\|W_{K_{\varepsilon_n},f}\|_{H^{m-1}(\Omega)}^2\geq C\,\capVm(K_{\varepsilon_n},f)\quad\text{for all } n.
	\end{equation}
	Let us consider
	$$Z_n:=\frac{W_{K_{\varepsilon_n},f}}{\|W_{K_{\varepsilon_n},f}\|_{H^{m-1}(\Omega)}}.$$
	We have
	\begin{equation*}
		\|Z_n\|_{H^{m-1}(\Omega)}=1\qquad\mbox{and}\qquad\|\nabla^mZ_n\|_2^2=\frac{\|\nabla^mW_{K_{\varepsilon_n},f}\|_2^2}{\|W_{K_{\varepsilon_n},f}\|^2_{H^{m-1}(\Omega)}}\leq\frac 1C
	\end{equation*}
	with $C>0$ as in \eqref{Lemma_A1_adapt_contrad}. 
	Then one may find a subsequence (still denoted by $Z_n$) and
        $Z\in V^m(\Omega)$, so that $Z_n\rightharpoonup Z$ in
        $V^m(\Omega)$. By the compact
        embedding $H^m(\Omega)\hookrightarrow\hookrightarrow
        H^{m-1}(\Omega)$, $Z$ is also the strong limit in the
        $H^{m-1}(\Omega)$ topology. This implies that
        $\|Z\|_{H^{m-1}(\Omega)}=1$. However, by the Mosco
        convergence of Lemma \ref{Mosco} one may show that
	\begin{equation}\label{Z=0}
		\int_{\Omega\setminus K}\nabla^mZ\,\nabla^m\varphi=0\qquad\mbox{for all}\;\varphi\in V^m_0(\Omega\setminus K),
              \end{equation}
              and hence for all $\varphi\in V^m(\Omega)$ 
	by Proposition \ref{Capacity_zero}, since we assumed
        $\capVm(K)=0$. Indeed, given $\varphi\in V^m_0(\Omega\setminus
        K)$ there exists a sequence
        $\{\varphi_{\varepsilon_n}\}_n$ so that
        $\varphi_{\varepsilon_n}\in V^m_0(\Omega\setminus
        K_{\varepsilon_n})$ for each $n\in\N$ and $\varphi_{\varepsilon_n}\to \varphi$ in $V^m(\Omega)$, for which then
	\begin{equation*}
		\int_{\Omega\setminus K_{\varepsilon_n}}\nabla^mZ_n\,\nabla^m\varphi_{\varepsilon_n}=0
	\end{equation*}
	for all $n\in\N$ by definition of $Z_n$ as a multiple of the capacitary potential $W_{K_{\varepsilon_n},f}$. Then, \eqref{Z=0} follows by weak-strong convergence in $V^m(\Omega)$, yielding $Z=0$, a contradiction.
\end{proof}

We are now in the position to prove the asymptotic expansion of the
perturbed eigenvalues. The suitable asymptotic parameter turns out to
be the $(u_J,V^m)$-capacity of the removed set, where
$u_J$ is an eigenfunction normalized in $L^2(\Omega)$ associated to the eigenvalue $\lambda_J$.

In the following, $(-\Delta)^m_\varepsilon$ stands for the polyharmonic operator acting on $V^m_0(\Omega\setminus K_\varepsilon)$. Similarly, to shorten notation, we write $\lambda_\varepsilon:=\lambda_J(\Omega\setminus K_\varepsilon)$ and the corresponding $(u_J,V^m)$-capacitary potential is denoted by $W_\varepsilon:=W_{K_\varepsilon,u_J}\in V^m(\Omega)$; we also write $\lambda_J:=\lambda_J(\Omega)$.

\begin{proof}[Proof of Theorem \ref{Prop_eigv_cap}]
	  First note that the simplicity of $\lambda_J$,
    i.e. of $\lambda_J(\Omega\setminus K)$ by Proposition
    \ref{Capacity_zero}, together with the convergence of the
    perturbed eigenvalues given by Theorem \ref{Alternative_conv_eigv},
    implies the simplicity of $\lambda_\varepsilon$ for $\varepsilon$
    sufficiently
    small.
  
	Let $\psi_\varepsilon:=u_J-W_\varepsilon\in V^m_0(\Omega\setminus K_\varepsilon)$ and $\varphi\in V^m_0(\Omega\setminus K_\varepsilon)$. Then
	\begin{equation*}
	\intOmega\nabla^m\psi_\varepsilon\nabla^m\varphi-\lambda_J\intOmega\psi_\varepsilon\varphi=\int_{\Omega\setminus K_\varepsilon}\!\nabla^m u_J\nabla^m\varphi-\lambda_J\intOmega\psi_\varepsilon\varphi=\lambda_J\intOmega W_\varepsilon\varphi.
	\end{equation*}
	This means that $\psi_\varepsilon$ satisfies weakly in
        $V^m_0(\Omega\setminus K_\varepsilon)$ the equation
	\begin{equation}\label{Prop_eigv_cap_eq}
          \left((-\Delta)^m-\lambda_J\right)\psi_\varepsilon=\lambda_JW_\varepsilon.
	\end{equation}
	Since by Lemma \ref{Lemma_A1_adapt} with $f=u_J$ one has
        $\|W_\varepsilon\|_2={\scriptstyle
          \mathcal{O}}(\capVm(K_\varepsilon,u_J)^{1/2})$ as $\varepsilon\to0$, we infer
	\begin{equation*}
	\dist(\lambda_J,\sigma((-\Delta)^m_\varepsilon))\leq\frac{\|((-\Delta)^m-\lambda_J)\psi_\varepsilon\|_2}{\|\psi_\varepsilon\|_2}={\scriptstyle \mathcal{O}}(\capVm(K_\varepsilon,u_J)^{1/2})
	\end{equation*}
        as $\varepsilon\to0$.   Since we know by Theorem
        \ref{Alternative_conv_eigv} and Proposition
        \ref{cap_pot_stability} that the spectrum of $(-\Delta)^m$ in
        $V^m_0(\Omega\setminus K_\varepsilon)$ varies continuously
        with respect to $\varepsilon$ and that the eigenvalue
        $\lambda_\varepsilon$ is simple for $\varepsilon$ small
        enough, one first deduces
        $$|\lambda_\varepsilon-\lambda_J|={\scriptstyle
          \mathcal{O}}(\capVm(K_\varepsilon,u_J)^{1/2})\quad\text{as
        }\varepsilon\to0.$$ Denote now by $\Pi_\varepsilon$ the
        projector (with respect to the scalar product in $L^2$) onto
        the eigenspace related to $\lambda_\varepsilon$ and take
        $u_\varepsilon:=\frac{\Pi_\varepsilon\psi_\varepsilon}{\|\Pi_\varepsilon\psi_\varepsilon\|_2}$
        as normalized eigenfunction. The first goal is to estimate the
        difference of the two eigenfunctions $u_J$ and
        $u_\varepsilon$:
	\begin{equation*}
	\begin{split}
	\|u_J-u_\varepsilon\|_2&\leq\|u_J-\psi_\varepsilon\|_2+\|\psi_\varepsilon-\Pi_\varepsilon\psi_\varepsilon\|_2+\left\|\Pi_\varepsilon\psi_\varepsilon-\frac{\Pi_\varepsilon\psi_\varepsilon}{\|\Pi_\varepsilon\psi_\varepsilon\|_2}\right\|_2\\
	&=\|W_\varepsilon\|_2+\|\psi_\varepsilon-\Pi_\varepsilon\psi_\varepsilon\|_2+\left|1-\|\Pi_\varepsilon\psi_\varepsilon\|_2^{-1}\right|\|\Pi_\varepsilon\psi_\varepsilon\|_2.
	\end{split}
	\end{equation*}
	Note that Lemma \ref{Lemma_A1_adapt} yields $\|W_\varepsilon\|_2={\scriptstyle \mathcal{O}}(\capVm(K_\varepsilon,u_J)^{1/2})$ and, moreover, we have
	$$\|\Pi_\varepsilon\psi_\varepsilon\|_2\leq\|\psi_\varepsilon\|_2\leq\|u_J\|_2+\|W_\varepsilon\|_2=\Oo(1)\quad\text{as
        }\varepsilon\to0.$$
	Hence, we need to estimate $\|\psi_\varepsilon-\Pi_\varepsilon\psi_\varepsilon\|_2$ and $\left|1-\|\Pi_\varepsilon\psi_\varepsilon\|_2^{-1}\right|$. We claim that both quantities are ${\scriptstyle \mathcal{O}}(\capVm(K_\varepsilon,u_J)^{1/2})$, obtaining thus
	\begin{equation}\label{uJ-ueps}
	\|u_J-u_\varepsilon\|_2={\scriptstyle
          \mathcal{O}}(\capVm(K_\varepsilon,u_J)^{1/2})\quad \text{as
        }\varepsilon\to0,
	\end{equation}
	and postpone the proof of such claim to the end of the
        proof. Then we have
	\begin{equation*}
	\begin{split}
	\capVm(K_\varepsilon,u_J)&=\intOmega|\nabla^m W_\varepsilon|^2=\intOmega\nabla^m(u_J-\psi_\varepsilon)\nabla^m W_\varepsilon=\intOmega\nabla^m u_J\nabla^m W_\varepsilon\\
	&=\lambda_J\intOmega u_JW_\varepsilon=\lambda_J\intOmega u_\varepsilon W_\varepsilon+\lambda_J\intOmega(u_J-u_\varepsilon)W_\varepsilon\\
	&\stackrel{\eqref{Prop_eigv_cap_eq}}{=}\intOmega\nabla^m\psi_\varepsilon\nabla^m u_\varepsilon-\lambda_J\intOmega u_\varepsilon\psi_\varepsilon+\lambda_J\intOmega(u_J-u_\varepsilon)W_\varepsilon\\
	&=\left(\lambda_\varepsilon-\lambda_J\right)\intOmega u_\varepsilon\psi_\varepsilon+\lambda_J\intOmega(u_J-u_\varepsilon)W_\varepsilon,
	\end{split}
	\end{equation*}
	and therefore
	\begin{equation}\label{lambdaJ-lambdaeps}
	\left(\lambda_\varepsilon-\lambda_J\right)\intOmega u_\varepsilon\psi_\varepsilon=\capVm(K_\varepsilon,u_J)-\lambda_J\intOmega\left(u_J-u_\varepsilon\right)W_\varepsilon.
	\end{equation}
	Since now
	\begin{equation*}
	\intOmega u_\varepsilon\psi_\varepsilon=\|u_\varepsilon\|_2^2+\intOmega u_\varepsilon\left(\psi_\varepsilon-u_\varepsilon\right)=1+\intOmega u_\varepsilon\left(\psi_\varepsilon-u_\varepsilon\right)
	\end{equation*}
	and
	\begin{equation*}
	\left|\intOmega u_\varepsilon\left(\psi_\varepsilon-u_\varepsilon\right)\right|\leq\|u_\varepsilon\|_2\|\psi_\varepsilon-u_\varepsilon\|_2={\scriptstyle \mathcal{O}}(\capVm(K_\varepsilon,u_J)^{1/2}),
	\end{equation*}
	where the last equality is again due to the claims above, from
        \eqref{lambdaJ-lambdaeps} and \eqref{uJ-ueps}, we infer
	\begin{equation*}
	\lambda_\varepsilon-\lambda_J=\frac{\capVm(K_\varepsilon,u_J)+{\scriptstyle \mathcal{O}}(\capVm(K_\varepsilon,u_J))}{1+{\scriptstyle \mathcal{O}}(1)}=\capVm(K_\varepsilon,u_J)\left(1+{\scriptstyle \mathcal{O}}(1)\right)
	\end{equation*}
        as $\varepsilon\to0$, as desired. To conclude, we prove the
        claims above. Since $\lambda_\varepsilon$ is a simple
        eigenvalue, denoting by $T_\varepsilon$ the restriction of
        $(-\Delta)^m_\varepsilon$ on $\ker\Pi_\varepsilon$, we have 
        that
        $\sigma(T_\varepsilon)=\sigma((-\Delta)^m_\varepsilon)\setminus\{\lambda_\varepsilon\}$
        and, by simplicity, 
        $\dist(\lambda_\varepsilon,\sigma(T_\varepsilon))\geq\delta$
for some $\delta>0$, uniformly with respect to  $\varepsilon$. Hence,
	\begin{equation*}
		\begin{split}
		\|\psi_\varepsilon-\Pi_\varepsilon\psi_\varepsilon\|_2&\leq\frac1\delta\left\|\left(T_\varepsilon-\lambda_\varepsilon\right)\left(\psi_\varepsilon-\Pi_\varepsilon\psi_\varepsilon\right)\right\|_2\les\|\left((-\Delta)^m-\lambda_\varepsilon\right)\psi_\varepsilon\|_2\\
		&\leq\|\left((-\Delta)^m-\lambda_J\right)\psi_\varepsilon\|_2+|\lambda_J-\lambda_\varepsilon|\|\psi_\varepsilon\|_2=|\lambda_J|\|W_\varepsilon\|_2+|\lambda_J-\lambda_\varepsilon|\|\psi_\varepsilon\|_2\\
		&={\scriptstyle \mathcal{O}}(\capVm(K_\varepsilon,u_J)^{1/2}).
		\end{split}
	\end{equation*}
	Since, by definition of $\psi_\varepsilon$ and Lemma
          \ref{Lemma_A1_adapt},
          $\|\psi_\varepsilon\|_2=1+{\scriptstyle
            \mathcal{O}(\capVm(K_\varepsilon,u_J)^{1/2})}$ as $\varepsilon\to0$,
          one thus finds that 
          $\|\Pi_\varepsilon\psi_\varepsilon\|_2=1+{\scriptstyle
            \mathcal{O}}(\capVm(K_\varepsilon,u_J)^{1/2})$, which in
          particular yields the desired estimate 
          $1-\|\Pi_\varepsilon\psi_\varepsilon\|_2^{-1}={\scriptstyle
            \mathcal{O}}(\capVm(K_\varepsilon,u_J)^{1/2})$. This concludes the proof.
\end{proof}

\begin{remark}
We observe that,	 in the proof of Theorem \ref{Prop_eigv_cap},
        the following estimate  for the normalized eigenfunction $u_\varepsilon\in
        V^m_0(\Omega\setminus K_\varepsilon)$
        of $(-\Delta)^m$ relative to $\lambda_J(\Omega\setminus
        K_\varepsilon)$  was established:
	$$\|u_\varepsilon-u_J\|_2={\scriptstyle
          \mathcal{O}}(\capVm(K_\varepsilon,u_J)^{1/2})\quad\text{as }\varepsilon\to0.$$
\end{remark}

\section{Sharp asymptotic expansions of perturbed eigenvalues: the
  case of uniformly shrinking holes.}
\subsection{A blow-up analysis.}\label{Sec_blowup}
In Theorem \ref{Prop_eigv_cap} we obtained an asymptotic expansion of
a perturbed simple eigenvalue in terms of $\capVm(K_\varepsilon,u_J)$,
in case the limiting removed set has zero $V^m$-capacity. However,
in view of possible applications, the dependence on the removed
set $K_\varepsilon$ is quite implicit using such an asymptotic
parameter. Therefore, we aim to understand how this quantity behaves
with respect to the diameter of the hole, in the case of a uniformly shrinking
family of compact sets which concentrate to a point, a set with
zero $V^m$-capacity in large dimensions by Proposition \ref{Capacity_point}.
\vskip0.2truecm First, we only suppose that
$\{K_\varepsilon\}_{\varepsilon>0}$ uniformly shrinks to a point,
which is assumed to be $0$ in the following, in the sense that
\begin{equation}\label{familiy_domain_unif_shr}
	K_\varepsilon\subset\overline{B_{C\varepsilon}(0)}
\end{equation}
for some constant $C>0$ and $\varepsilon$ small enough. The following
is a generalization of \cite[Lemma 2.2]{AFHL} to the higher-order
setting.
\begin{prop}\label{lem_shrinking}
  Let $\Omega\subset\R^N$ be a smooth bounded domain such that
  $0\in\Omega$ and let $\{K_\varepsilon\}_{\varepsilon>0}$ be a family
  of compact sets satisfying \eqref{familiy_domain_unif_shr}. Let
  $h\in H^m(\Omega)$ be such that $|D^kh(x)|=\Oo(|x|^{\gamma-k})$ as
  $|x|\to0$ for some $\gamma\in\N$ and all $k\in\{0,\dots,m\}$. Then
	\begin{equation}\label{lem_shrinking_asympt}
		\capVm(K_\varepsilon,h)=\Oo(\varepsilon^{N-2m+2\gamma})\quad\text{as
                $\varepsilon\to0$}.
	\end{equation}
\end{prop}
\begin{proof}
	By Proposition \ref{Capacity_monotonicity}(\textit{iii}), it
        is sufficient to prove \eqref{lem_shrinking_asympt} for the
        Dirichlet case $V^m(\Omega)=H^m_0(\Omega)$. Let $\varphi\in
        C^\infty_0(\R^N)$ with $\supp\,\varphi\subset B_2(0)$ and
        $\varphi\equiv1$ in a neighbourhood of $\overline{B_1(0)}$, and define
        $\varphi_\varepsilon(x):=\varphi((C\varepsilon)^{-1}x)$
        for all $\varepsilon>0$ small. Then
        $h_\varepsilon:=\varphi_\varepsilon h$ coincides with $h$ 
in a neighbourhood of  $\overline{B_{C\varepsilon}(0)}$. By monotonicity
\begin{equation*}
		\begin{split}
			\capm(K_\varepsilon,h)&\leq\capm(\overline{B_{C\varepsilon}(0)},h)\leq\intOmega|\nabla^m h_\varepsilon|^2\\
			&\les\sum_{k=0}^m\int_{B_{2C\varepsilon}(0)}|D^{m-k}\varphi_\varepsilon(x)|^2|D^kh(x)|^2\,dx\\
			&\les\sum_{k=0}^m(C\varepsilon)^{2k-2m}\int_{B_{2C\varepsilon}(0)}\left|D^{m-k}\varphi\left(\frac x{C\varepsilon}\right)\right|^2|D^kh(x)|^2\,dx\\
			&\les\sum_{k=0}^m(C\varepsilon)^{2k-2m+N}\int_{B_2(0)}|D^{m-k}\varphi(y)|^2|D^kh(C\varepsilon y)|^2\,dy\\
			&\les\varepsilon^{N-2m+2\gamma}\sum_{k=0}^m\int_{B_2(0)}|D^{m-k}\varphi(y)|^2\,dy\les\varepsilon^{N-2m+2\gamma},
		\end{split}
	\end{equation*}
	having used the assumption that
        $\|D^kh\|_\infty\les\varepsilon^{\gamma-k}$ in 
$B_{2C\varepsilon}(0)$.
\end{proof}

Next, having in mind the model case $K_\varepsilon:=\varepsilon\cK$ for a fixed compactum $\cK$, we consider families of compact sets which uniformly shrink to $\{0\}$ as in \eqref{familiy_domain_unif_shr} but enjoying a more specific structure. 
To this aim we assume
\begin{enumerate}[(M1)]
	\item there exists $M\subset\R^N$ compact such that $\varepsilon^{-1}K_\varepsilon\subseteq M$ for all $\varepsilon\in(0,1)\,$;
	\item there exists $\cK\subset\R^N$ compact such that $\R^N\setminus\varepsilon^{-1}K_\varepsilon\to\R^N\setminus\cK$ in the sense of Mosco as $\varepsilon\to0$.
\end{enumerate}
In our context (M2) means the following:
\begin{enumerate}
	\item[(\textit{i})] if $u_\varepsilon\in D^{m,2}_0(\R^N\setminus\varepsilon^{-1}K_\varepsilon)$ is so that $u_\varepsilon\rightharpoonup u$ in $D^{m,2}_0(\R^N)$ as $\varepsilon\to0$, then $u\in D^{m,2}_0(\R^N\setminus\cK)$;
	\item[(\textit{ii})] if $u\in D^{m,2}_0(\R^N\setminus\cK)$,
          then there exists a family
          $\{u_\varepsilon\}_{\varepsilon>0}$ such that
            $u_\varepsilon\in
            D^{m,2}_0(\R^N\setminus\varepsilon^{-1}K_\varepsilon)$ for
            all $\varepsilon>0$ and $u_\varepsilon\to u$ in $D^{m,2}_0(\R^N)$ as $\varepsilon\to0$.
\end{enumerate}
In this case we  also say that
$D^{m,2}_0(\R^N\setminus\varepsilon^{-1}K_\varepsilon)$ converges to
$D^{m,2}_0(\R^N\setminus\cK)$ in the sense of Mosco.

\begin{remark}\label{rmk_MK}
  Assumption (M1) is actually equivalent to the condition
  \eqref{familiy_domain_unif_shr}, since $M\subset B_C(0)$ for some $C>0$.
 \end{remark}

\begin{lem}\label{Mosco_equiv}
  Let $N>2m$. Under the assumption (M1) the following are equivalent:
	\begin{enumerate}
		\item $D^{m,2}_0(\R^N\setminus\varepsilon^{-1}K_\varepsilon)$ converges to $D^{m,2}_0(\R^N\setminus\cK)$ in the sense of Mosco;
		\item $H^m_0(B_R(0)\setminus\varepsilon^{-1}K_\varepsilon)$ converges to $H^m_0(B_R(0)\setminus\cK)$ in the sense of Mosco for all $R>r(M)$, where $r(M):=\inf\{\rho>0\,|\,B_\rho(0)\supset M\}$.
	\end{enumerate}
\end{lem}
We denote by (M$_R$2.i) and (M$_R$2.ii) the correspondent conditions (M2.i) and (M2.ii) which enter in the definition of the Mosco convergence relative to the space $H^m_0(B_R(0))$. In the following we use the shorter notation $B_R:=B_R(0)$.
\begin{proof}
  $\boldsymbol{1)\Rightarrow 2)}$. \textit{Verification of
    (M$_R$2.i)}. Let
  $\{u_\varepsilon\}_{\varepsilon>0}\subset H^m_0(B_R)$ be a family of
  functions such that $u_\varepsilon\in H^m_0(B_R\setminus\eKe)$ and
  $u_\varepsilon\rightharpoonup u$ in $H^m_0(B_R)$. We show that
  $u\in H^m_0(B_R\setminus~\!\cK)$. Denoting by $u_\varepsilon^E$ and
  $u^E$ the trivial extension of $u_\varepsilon$ and $u$ outside $B_R$
  respectively, then $u_\varepsilon^E\in D^{m,2}_0(\R^N\setminus\eKe)$
  and $u_\varepsilon^E\rightharpoonup u^E$ in $D^{m,2}_0(\R^N)$. Hence
  condition (M2.i) guarantees that
  $u^E\in D^{m,2}_0(\R^N\setminus\cK)$, which by construction implies
  that $u\in H^m_0(B_R\setminus\cK)$.
	 
	 \textit{Verification of (M$_R$2.ii)}. Let
         $v\in C^\infty_0(B_R\setminus\cK)$ and
         $\Lambda_1,\Lambda_2\subset\Omega$ be two open sets such that
         $\supp\,v\subset\subset\Lambda_1\subset\subset\Lambda_2\subset\subset
         B_R$. Take $\eta\in C^\infty_0(\Lambda_2)$ with $\eta\equiv1$
         on $\Lambda_1$. Since $v^E\in D^{m,2}_0(\R^N\setminus\cK)$,
         then by (M2.ii) there exists a family 
         $\{v_\varepsilon\}_{\varepsilon>0}\subset D^{m,2}_0(\R^N)$ with
         $v_\varepsilon\in D^{m,2}_0(\R^N\setminus\eKe)$ for all
         $\varepsilon>0$ such that $v_\varepsilon\to v^E$ in
         $D^{m,2}_0(\R^N)$,
         i.e. $\|\nabla^m(v_\varepsilon-v^E)\|_{L^2(\R^N)}\to0$ as
         $\varepsilon\to0$. By construction,
         $\eta v_\varepsilon\in H^m_0(B_R\setminus\eKe)$. We claim
         that $\eta v_\varepsilon\to v$ in $H^m_0(B_R)$. Indeed,
         denoting by $q_j:=2^*_{m,j}=\frac{2N}{N-2(m-j)}\geq2$ and
         $p_j:=2\left(\tfrac{q_j}2\right)'=\frac N{m-j}$ for $j\in\{0,\dots,m\}$,
         one has
	 \begin{equation}\label{eq:1}
	 	\begin{split}
	 		\|\nabla^m(\eta v_\varepsilon-v)\|_{L^2(B_R)}&=\|\nabla^m(\eta(v_\varepsilon-v))\|_{L^2(B_R)}\leq\sum_{j=0}^m\|D^{m-j}\eta\,D^j(v_\varepsilon-v)\|_{L^2(B_R)}\\
	 		&\leq\sum_{j=0}^m\|D^{m-j}\eta\|_{L^{p_j}(\supp\,\eta)}\|D^j(v_\varepsilon-v^{E})\|_{L^{q_j}(\R^N)}\\
	 		&\leq\sum_{j=0}^m|\supp\,\eta|^{\frac1{p_j}}\|D^{m-j}\eta\|_\infty\|D^m(v_\varepsilon-v^{E})\|_{L^2(\R^N)}\\
	 		&\leq C(m,N,R)\|\eta\|_{W^{m,\infty}(\R^N)}\|\nabla^m(v_\varepsilon-v^{E})\|_{L^2(\R^N)}\to0.
	 	\end{split}
	 \end{equation}
	 The last steps are due to the critical Sobolev embedding on
         $\R^N$ (for which it is fundamental that $N>2m$), see \eqref{eq:sob-ineq}, and to the equivalence of the norms $\|D^m\cdot\|_2$ and $\|\nabla^m\cdot\|_2$, see e.g. \cite[Chp. 2.2]{GGS}.

The above argument and the density of $C^\infty_0(B_R\setminus\cK)$
  in $H^m_0(B_R\setminus\cK)$ imply that, fixing any $v\in
  H^m_0(B_R\setminus\cK)$,
  for every $\delta>0$ there
  exists a family 
  $\{v_{\delta,\varepsilon}\}_{\varepsilon>0}$ such that
  $v_{\delta,\varepsilon}\in H^m_0(B_R\setminus K_\varepsilon)$ and 
  $\|v_{\delta,\varepsilon}-v\|_{H^m(B_R)}<\delta$ for all
  $\varepsilon\in(0,\bar\varepsilon_\delta]$ for some
  $\bar\varepsilon_\delta>0$. Therefore there exists a vanishing sequence
  $(\varepsilon_n)_n\searrow0$ such that  $\big\|v_{\frac1k,\varepsilon}-v\big\|_{H^m(B_R)}<\frac1k$ for all
  $\varepsilon\in(0,\varepsilon_k]$. Defining
  $v_\varepsilon=v_{\frac1n,\varepsilon}$ for
  $\varepsilon\in(\varepsilon_{n+1}, \varepsilon_{n}]$, we have that,
  for all $\varepsilon\in (0, \varepsilon_{1}]$, $v_\varepsilon\in
  H^m_0(B_R\setminus K_\varepsilon)$ and $v_\varepsilon\to v$ in
  $H^m_0(B_R)$ as $\varepsilon\to0$.
         
	 $\boldsymbol{2)\Rightarrow 1)}$. \textit{Verification of
           (M2.i)}. Let
         $\{u_\varepsilon\}_{\varepsilon>0}\subset D^{m,2}_0(\R^N)$ be
         a family of functions such that
         $u_\varepsilon\in D^{m,2}_0(\R^N\setminus\eKe)$ and
         $u_\varepsilon\rightharpoonup u$ in $D^{m,2}_0(\R^N)$. Taking
         $\eta\in C^\infty(\R^N)$ and $R>0$ such that
         $\supp\,\eta\subset B_R(0)$, due to the continuity of
           the map $D^{m,2}_0(\R^N)\to H^m_0(B_R)$, $u\mapsto \eta u$,
           which can be easily proved arguing as in \eqref{eq:1}, one
         has that $\eta u_\varepsilon\rightharpoonup\eta u$ in
         $H^m_0(B_R)$. Hence, (M$_R$2.i) implies
         $\eta u\in H^m_0(B_R\setminus\cK)$. Hence
           $\eta u\in D^{m,2}_0(\R^N\setminus\cK)$ for every
           $\eta\in C^\infty_0(\R^N)$. Let us now take
         $\eta_1\in C^\infty_0(\R^N)$ with $0\leq\eta_1\leq1$,
         $\eta_1\equiv1$ on $B_{\frac12}$ and
         $\supp\,\eta_1\subset B_1$, and define
         $\eta_k:=\eta_1\left(\frac\cdot k\right)$, so that
         $\supp\,\eta_k\subset B_k$. We are going to prove that
         $\eta_ku\to u$ in
         $D^{m,2}_0(\R^N)$ as $k\to+\infty$, in order to conclude that
         $u\in D^{m,2}_0(\R^N\setminus\cK)$. We estimate as follows:
\begin{equation}\label{eq:2}
		\begin{split}
	 		\|\nabla^m(\eta_ku-u)\|_{L^2(\R^N)}^2&\leq\int_{\R^N}|\eta_k-1|^2|\nabla^mu|^2+\sum_{j=0}^{m-1}\int_{B_k\setminus
                          B_{\frac k2}}|D^{m-j}\eta_k|^2|D^ju|^2\\
	 		&\leq\int_{\R^N\setminus B_{\frac k2}}|\nabla^mu|^2+\sum_{j=0}^{m-1}\|D^{m-j}\eta_k\|_{L^{p_j}(\R^N)}^2\|D^ju\|_{L^{q_j}(\R^N\setminus B_{\frac k2})}^{2}\\
	 		&=\int_{\R^N\setminus B_{\frac k2}}|\nabla^mu|^2+\sum_{j=0}^{m-1}\|D^{m-j}\eta_1\|_{L^{p_j}(\R^N)}^2\|D^ju\|_{L^{q_j}(\R^N\setminus B_{\frac k2})}^{2},
	 	\end{split}
 	\end{equation}
	where
         we have used the fact that $\supp\, (D^{m-j}\eta_k)\subset B_k\setminus
                          B_{\frac k2}$ if $j\leq m-1$ and 
          \begin{align*}
            \|D^{m-j}\eta_k\|_{L^{p_j}(\R^N)}^2&=k^{-2(m-j)}\left(\int_{\R^N}|D^{m-j}\eta_1(x/k)|^{\frac{N}{m-j}}\,dx\right)^{\!\!\frac{2(m-j)}N}\\&=\left(\int_{\R^N}|D^{m-j}\eta_1(y)|^{\frac{N}{m-j}}\,dy\right)^{\!\!\frac{2(m-j)}N}.
          \end{align*}
     Since $u\in D^{m,2}_0(\R^N)$, the first term at the right-hand
     side of \eqref{eq:2} converges to $0$ as $k\to+\infty$;
moreover, the critical Sobolev embedding
        \eqref{eq:sob-ineq} implies that $D^j u\in
          L^{q_j}(\R^N)$ for all $0\leq j\leq m$, so that also the
          second term goes to $0$. We conclude that $\eta_ku\to u$ in $D^{m,2}_0(\R^N)$ as $k\to+\infty$, which yields $u\in D^{m,2}_0(\R^N\setminus\cK)$.

        \textit{Verification of (M2.ii)}. Let
        $u\in D^{m,2}_0(\R^N\setminus\cK)$. Let $\delta>0$. By
        density, there exists a function
        $v\in C^\infty_0(\R^N\setminus\cK)$ such that
        $\|\nabla^m(u-v)\|_{L^2(\R^N)}<\frac\delta2$. Take $R>0$ so
        that $\supp\,v\subset B_R$. Then
        $v\in H^m_0(B_R\setminus\cK)$ and by (M$_R$2.ii) there
        exist $\bar\varepsilon_\delta>0$ and a family of functions
          $\{\varphi^\delta_\varepsilon\}_{\varepsilon\in(0, \bar\varepsilon_\delta)}$
          such that $\varphi^\delta_\varepsilon\in H^m_0(B_R\setminus\eKe)$
          and
          $\|\nabla^m(v-\varphi^\delta_\varepsilon)\|_{L^2(B_R)}<\frac\delta2$
          for all $\varepsilon\in(0, \bar\varepsilon_\delta)$. Hence, for
          all $\varepsilon\in(0, \bar\varepsilon_\delta)$,
          $(\varphi_\varepsilon^\delta)^E\in D^{m,2}_0(\R^N\setminus\eKe)$ and
          $\|\nabla^m(u-(\varphi_\varepsilon^\delta)^E)\|_{L^2(\R^N)}<\delta$.

          As a consequence, there exists a strictly decreasing and
          vanishing sequence 
          $\{\varepsilon_n\}_n$ such that, for every $n\in{\mathbb
            N}\setminus\{0\}$, there exists a family of functions
          $\{u^n_\varepsilon\}_{\varepsilon\in(0,\varepsilon_n)}$
          such that $u^n_\varepsilon\in D^{m,2}_0(\R^N\setminus\eKe)$
          and $\|\nabla^m(u-u_\varepsilon^n)\|_{L^2(\R^N)}<\frac1n$ for
          all $\varepsilon\in(0, \varepsilon_n)$. For every
          $\varepsilon\in (0,\varepsilon_1)$, we define
          $u_\varepsilon:=u_\varepsilon^n$ if
          $\varepsilon_{n+1}\leq\varepsilon<\varepsilon_n$. It is easy
          to verify that,  by construction, $u_\varepsilon\in
          D^{m,2}_0(\R^N\setminus\eKe)$ for all $\varepsilon\in
          (0,\varepsilon_1)$ and
          $\|\nabla^m(u-u_\varepsilon)\|_{L^2(\R^N)}\to 0$ as
          $\varepsilon\to 0$. (M2.ii) is thereby verified.        
      \end{proof}

Before stating the main results of the section, we prepose a lemma about the stability of the $(h,V^m)$-capacitary potential with respect to the function $h$.
\begin{lem}\label{conv_cap_pot}
  Let $K\subset\Omega\subset\R^N$, $K$ compact,
  $\{h_n\}_{n\in{\mathbb N}}\subset H^m_{loc}(\Omega)$ and $h\in
  H^m_{loc}(\Omega)$. Let us suppose that, for some $\mathcal
  U(K)\subset\Omega$ open neighbourhood of $K$, $h_n\to h$ in $H^m(\mathcal
  U(K))$ as $n\to\infty$ and
  denote by $W_{K,h_n}$ (resp. $W_{K,h}$) the capacitary potential for
  $\capVm(K,h_n)$ (resp. $\capVm(K,h)$). Then $W_{K,h_n}\to W_{K,h}$
  in $V^m(\Omega)$ and
  $\capVm(K,h_n)\to\capVm(K,h)$.
\end{lem}
\begin{proof}
Being capacitary potentials, the functions $W_{K,h_n}$ and $W_{K,h}$ satisfy
\begin{equation*}
\intOmega\nabla^m W_{K,h}\nabla^m\varphi=0\quad\text{and}\quad
\intOmega\nabla^m W_{K,h_n}\nabla^m\varphi=0\quad \text{for all }n\in
{\mathbb N}\text{ and } \varphi\in V^m_0(\Omega\setminus K).
\end{equation*}
Let $\eta_K\in C^\infty(\R^N)$ be a cutoff function such that
$0\leq \eta_K\leq1$, $\supp\,\eta_K\subset \cU(K)$ and $\eta_K\equiv 1$ in a
  neighbourhood  of $K$. Hence, by construction, one has that
	$$W_{K,h_n}-\eta_Kh_n\in V^m_0(\Omega\setminus K)\qquad\mbox{and}\qquad W_{K,h}-\eta_Kh\in V^m_0(\Omega\setminus K).$$
	Therefore
	\begin{equation*}
	\begin{split}
		\|\nabla^m(W_{K,h_n}&-W_{K,h})\|_2^2=\intOmega\left(\nabla^mW_{K,h_n}-\nabla^mW_{K,h}\right)\left(\nabla^mW_{K,h_n}-\nabla^mW_{K,h}\right)\\
		&=\intOmega\nabla^mW_{K,h_n}\nabla^m\left(W_{K,h_n}-\eta_Kh_n\right)+\intOmega\nabla^mW_{K,h_n}\nabla^m\left(\eta_Kh_n-\eta_Kh\right)\\
		&\quad+\intOmega\nabla^mW_{K,h_n}\nabla^m\left(\eta_Kh-W_{K,h}\right)-\intOmega\nabla^mW_{K,h}\nabla^m\left(W_{K,h_n}-\eta_Kh_n\right)\\
		&\quad-\intOmega\nabla^mW_{K,h}\nabla^m\left(\eta_Kh_n-\eta_Kh\right)-\intOmega\nabla^mW_{K,h}\nabla^m\left(\eta_Kh-W_{K,h}\right)\\
		&=\intOmega\nabla^m\left(W_{K,h_n}-W_{K,h}\right)\nabla^m\left(\eta_Kh_n-\eta_Kh\right)\\
		&\leq\|\nabla^mW_{K,h_n}-\nabla^mW_{K,h}\|_2\|\nabla^m\left(\eta_K\left(h_n-h\right)\right)\|_2.
	\end{split}
	\end{equation*}
	This yields
	\begin{equation*}
	\begin{split}
		\|\nabla^mW_{K,h_n}-\nabla^mW_{K,h}\|_2&\leq\|\nabla^m\left(\eta_K\left(h_n-h\right)\right)\|_2
		\les\|h_n-h\|_{H^m(\cU(K))}\to0,
	\end{split}
	\end{equation*}
        i.e. $W_{K,h_n}\to W_{K,h}$ in $V^m(\Omega)$, directly
        implying that $\capVm(K,h_n)\to\capVm(K,h)$ as $n\to\infty$.
\end{proof}

\begin{remark}
	In case $\Omega=\R^N$ the same result holds with $W_{K,h_n}\to W_{K,h}$ in $D^{m,2}_0(\R^N)$.
\end{remark}

We are now in the position to prove the main results of this section,
namely a generalized version of Theorems
\ref{Thm_blowup_easycase}-\ref{asymptotic_eigv_easycase_Nav}, which
take into account families of domains which satisfy (M1)-(M2), rather
than just the model case $K_\varepsilon=\varepsilon\cK$.

Motivated by the asymptotic scaling properties of the eigenfunctions
\eqref{claim_U}, we apply a blow-up argument to a 
  rescaled problem, 
in order to find a limit equation on $\R^N\setminus\cK$ and to
prove the convergence of the family of scaled capacitary potentials
to the one for the limiting problem. The capacity
$\capVm(K_\varepsilon,u_J)$ will behave then as the limit capacity on
$\R^N\setminus\cK$ multiplied by a suitable power of $\varepsilon$
given by the scaling. In this argument, we work with
  the homogeneous Sobolev spaces and, in particular, for the Navier case
  the characterization via Hardy-Rellich inequalities of Section
  \ref{Sec_Hardy} will be needed. This is the main reason for the
  restriction to the fourth-order case in the Navier setting, since, up to our knowledge,
  the extension of Proposition \ref{HR_spaces} to the full generality
  $m\geq2$ is an open problem.  \vskip0.2truecm
\begin{thm}[Asymptotic expansion of the capacity, Dirichlet case]\label{Thm_blowup}
  Let $N>2m$ and $\Omega\subset\R^N$ be a bounded smooth domain with
  $0\in\Omega$. Let $\{K_\varepsilon\}_{\varepsilon>0}$ be a
    family of compact sets  uniformly concentrating to $\{0\}$
    satisfying   (M1)-(M2) for some compact set $\cK$. Let $\lambda_J$ be an eigenvalue of
  \eqref{eq_unperturbed} with Dirichlet boundary conditions and $u_J\in H^m_0(\Omega)$ be a corresponding
  eigenfunction normalized in $L^2(\Omega)$. Then
	\begin{equation}\label{Thm_blowup_asympt}
		{\rm cap}_{m\!,\,\Omega}(K_\varepsilon,u_J)=\varepsilon^{N-2m+2\gamma}\left(\capp_{m,\R^N}(\cK,U_0)+{\scriptstyle \mathcal{O}}(1)\right)
	\end{equation}
	as $\varepsilon\to0$, with $\gamma$ and $U_0$ as in \eqref{claim_U}.
\end{thm}
\begin{thm}[Asymptotic expansion of the capacity, Navier case]\label{Thm_blowup_Nav}
  Let $N>4$ and $\Omega\subset\R^N$ be a bounded smooth domain with
  $0\in\Omega$.  Let $\{K_\varepsilon\}_{\varepsilon>0}$ be a
    family of compact sets uniformly concentrating to $\{0\}$
    satisfying (M1)-(M2) for some compact set $\cK$. Let $\lambda_J$
  be an eigenvalue of \eqref{eq_unperturbed} with Navier boundary
    conditions and $u_J\in H^2_\vartheta(\Omega)$ be a corresponding
  eigenfunction normalized in $L^2(\Omega)$. Then
	\begin{equation}\label{Thm_blowup_asympt_Nav}
		{\rm cap}_{2, \vartheta\!,\,\Omega}(K_\varepsilon,u_J)=\varepsilon^{N-4+2\gamma}\left(\capp_{2,\R^N}(\cK,U_0)+{\scriptstyle \mathcal{O}}(1)\right)
	\end{equation}
	as $\varepsilon\to0$, with $\gamma$ and $U_0$ as in \eqref{claim_U} with $m=2$.
\end{thm}

As a direct consequence, braiding together Theorem \ref{Prop_eigv_cap}
and Theorems \ref{Thm_blowup}-\ref{Thm_blowup_Nav} respectively, and
recalling that for $N\geq2m$ the point has null $V^m$-capacity by
Proposition \ref{Capacity_point}, we obtain Theorems
  \ref{asymptotic_exp_eigv_cap} and \ref{asymptotic_exp_eigv_cap_Nav} below.
\begin{thm}[Asymptotic expansion of perturbed eigenvalues, Dirichlet case]\label{asymptotic_exp_eigv_cap}
  Let $N>2m$ and $\Omega\subset\R^N$ be a bounded smooth domain
  containing $0$.  Let $\{K_\varepsilon\}_{\varepsilon>0}$ be a
    family of compact sets uniformly concentrating to $\{0\}$
    satisfying (M1)-(M2) for some compact set $\cK$. Let $\lambda_J$
  be a simple eigenvalue of \eqref{eq_unperturbed} with Dirichlet
    boundary conditions and let $u_J\in H^m_0(\Omega)$ be a
  corresponding eigenfunction normalized in $L^2(\Omega)$. Then
	\begin{equation}\label{expansion_eigv}
		\lambda_J(\Omega\setminus K_\varepsilon)=\lambda_J(\Omega)+\varepsilon^{N-2m+2\gamma}\left(\capp_{m,\R^N}(\cK,U_0)+{\scriptstyle \mathcal{O}}(1)\right)
	\end{equation}
	as $\varepsilon\to0$, with $\gamma$ and $U_0$ as in \eqref{claim_U}.
\end{thm}
\begin{thm}[Asymptotic expansion of perturbed eigenvalues, Navier case]\label{asymptotic_exp_eigv_cap_Nav}
  Let $N>4$ and $\Omega\subset~\!\!\R^N$ be a bounded smooth domain
  containing $0$.
Let $\{K_\varepsilon\}_{\varepsilon>0}$ be a
    family of compact sets uniformly concentrating to $\{0\}$
    satisfying (M1)-(M2) for some compact set $\cK$.
Let $\lambda_J$ be a simple eigenvalue of
  \eqref{eq_unperturbed}  with Navier
    boundary conditions and let $u_J\in H^2_\vartheta(\Omega)$ be a
  corresponding eigenfunction normalized in $L^2(\Omega)$. Then
	\begin{equation}\label{expansion_eigv_Nav}
		\lambda_J(\Omega\setminus K_\varepsilon)=\lambda_J(\Omega)+\varepsilon^{N-4+2\gamma}\left(\capp_{2,\R^N}(\cK,U_0)+{\scriptstyle \mathcal{O}}(1)\right)
	\end{equation}
	as $\varepsilon\to0$, with $\gamma$ and $U_0$ as in \eqref{claim_U} with $m=2$.
\end{thm}

The proofs of Theorems \ref{Thm_blowup} and \ref{Thm_blowup_Nav} follow a similar structure. We proceed hence to prove them at once using the introduced unifying notation, detailing the differences when needed.

\begin{proof}[Proof of Theorems \ref{Thm_blowup}-\ref{Thm_blowup_Nav}]
	Motivated by \eqref{claim_U}, we define the analogously scaled potentials
	$$\widetilde W_\varepsilon:=\frac{W_\varepsilon(\varepsilon\cdot)}{\varepsilon^\gamma},\qquad\mbox{where}\quad W_\varepsilon:=W_{K_\varepsilon,u_J}.$$
	It is easy to verify that $\tW_\varepsilon$ is the capacitary potential for $U_\varepsilon$ in $\varepsilon^{-1}\Omega\setminus\eKe$, i.e.
	\begin{equation}\label{tW}
		\begin{cases}
			(-\Delta)^m\tW_\varepsilon=0\quad\mbox{in}\;\varepsilon^{-1}\Omega\setminus\eKe,\\
			\tW_\varepsilon\in V^m(\varepsilon^{-1}\Omega),\\
			\tW_\varepsilon-U_\varepsilon\in V^m_0(\varepsilon^{-1}\Omega\setminus\eKe),
		\end{cases}
	\end{equation}
	where $m=2$ in the Navier case. The first goal now is to prove
        that the so-rescaled capacitary potentials weakly converge to
        some function $\tW$ and prove that $\tW$ is a capacitary
        potential in $\R^N\setminus\cK$. To this aim, we need to
        distinguish between Dirichlet and Navier conditions on
        $\dOmega$. Indeed, by extension by zero outside the rescaled
        domains, in the first case it is rather natural to prove that
        the limit functional space is $D^{m,2}_0(\R^N\setminus\cK)$;
        that the same holds true in the Navier case is not evident and
        requires a finer analysis. A fundamental role in
        this second case is played by the Hardy-Rellich inequality
        discussed in Section \ref{Sec_Hardy}, which is however
        available just for $m=2$.  The two cases will converge then in
        the final step where the asymptotic expansions
      \eqref{Thm_blowup_asympt}--\eqref{Thm_blowup_asympt_Nav} are proved.

	\vskip0.2truecm
	\indent \textbf{Step 1 (Dirichlet case
          $\boldsymbol{V^m=H^m_0}$).}
        By (M1)  there exists
        $R>0$ such that, for $\varepsilon$ small enough, $\eKe\subset
      M\subset B_R(0)\subset\varepsilon^{-1}\Omega$, and hence,
      in view of Proposition \ref{Capacity_monotonicity},
	$$\capp_{m,\varepsilon^{-1}\Omega}(\eKe,U_\varepsilon)\leq\capp_{m,B_R(0)}(M,U_\varepsilon).$$
	Since $U_\varepsilon\to U_0$ in $H^m(B_R(0))$ by
        \eqref{claim_U}, applying Lemma \ref{conv_cap_pot} in
        $B_R(0)$, we infer that 
	$$\capp_{m,B_R(0)}(M,U_\varepsilon)\to\capp_{m,B_R(0)}(M,U_0)\quad\text{as
        }\varepsilon\to0.$$ This yields in particular that
        $\|\nabla^m\tW_\varepsilon\|_{L^2(\varepsilon^{-1}\Omega)}^2=\capp_{m,\varepsilon^{-1}\Omega}(\eKe,U_\varepsilon)$
        is bounded uniformly with respect to
        $\varepsilon$. Letting  $\tWEe$ be the
        extension by $0$ of $\tW_\varepsilon$ outside
        $\varepsilon^{-1}\Omega$, we have thus that 
        $\|\tWEe\|_{D^{m,2}_0(\R^N)}\leq C$. Since $D^{m,2}_0(\R^N)$
        is a Hilbert space, and so
        reflexive, 
for every sequence
  $\varepsilon_n\to0^+$ there exist a subsequence $\varepsilon_{n_k}$
  and $\tW\in D^{m,2}_0(\R^N)$ such that
	\begin{equation}\label{conv_tWe_DIR}
\tW^E_{\varepsilon_{n_k}}\rightharpoonup\tW\quad\mbox{weakly in }
D^{m,2}_0(\R^N)\text{ as }k\to\infty.
	\end{equation}
	We claim now that
        $\|\nabla^m\tW\|_2^2=\capp_{m,\R^N}(\cK,U_0)$.
	
	Let $\varphi\in C^\infty_0(\R^N\setminus\cK)$ and $R>r(M)$ be
        such that $\supp\,\varphi\subset B_R$, then by
        (M$_R$2-\textit{ii}) of Lemma \ref{Mosco_equiv}, one may find
        a family
        $\{\varphi_\varepsilon\}_{\varepsilon>0}\subset H^m_0(B_R)$
        such that $\varphi_\varepsilon\in H^m_0(B_R\setminus\eKe)$ and
        $\varphi_\varepsilon\to\varphi$ in $H^m_0(B _R)$ as
        $\varepsilon\to0$. In particular, for $\varepsilon$ small
        enough, one has that $B_R\subset\varepsilon^{-1}\Omega$, so
        $\varphi_\varepsilon$ may be taken as test function for the
        capacitary potential $\tW_\varepsilon$. Hence,
	\begin{align*}
	0&=\int_{\varepsilon_{n_k}^{-1}\Omega\setminus
          \varepsilon_{n_k}^{-1}K_{\varepsilon_{n_k}}}\nabla^m\widetilde
        W_{
          \varepsilon_{n_k}}\,\nabla^m\varphi_{\varepsilon_{n_k}}\\
        &=\int_{\R^N\setminus
          \varepsilon_{n_k}^{-1}K_{\varepsilon_{n_k}}}
\nabla^m\widetilde
        W^E_{ \varepsilon_{n_k}}\,\nabla^m\varphi_{\varepsilon_{n_k}}
       \to\int_{\R^N\setminus\cK}\nabla^m\tW\,\nabla^m\varphi\quad\text{as
          }k\to\infty
        \end{align*}
	by weak-strong convergence in $D^{m,2}_0(\R^N)$. 
	We are left to show that $\tW-\eta U_0\in
        D^{m,2}_0(\R^N\setminus\cK)$, for some cutoff
        function $\eta$ which is equal to $1$ in a neighbourhood of
        $\cK$. Let $\eta\in C^\infty_0(\R^N)$ be equal to $1$ on an
          open set $\mathcal U$ with $\cK\cup M\subset\mathcal U$;
        hence $\eta$ is also equal to $1$ on
        neighbourhoods of each $\eKe$ by (M1). Then
$\tW_\varepsilon^E-\eta U_\varepsilon\in
  D^{m,2}_0(\R^N\setminus\eKe)$ and
        $\tW^E_{\varepsilon_{n_k}}-\eta
        U_{\varepsilon_{n_k}}\rightharpoonup\tW-\eta U_0$ in
        $D^{m,2}_0(\R^N)$ as $k\to\infty$, and so by
        (M2-\textit{i}) one infers that $\tW-\eta U_0\in
        D^{m,2}_0(\R^N\setminus\cK)$.
	All in all, we deduce that $\tW$ is the capacitary potential relative to $\capp_{m,\R^N}(\cK,U_0)$, i.e.
	\begin{equation}\label{cap_Rn}
		\|\nabla^m\tW\|^2_{L^2(\R^N)}=\capp_{m,\R^N}(\cK,U_0).
	\end{equation}
Since the limit $\tW$ in \eqref{conv_tWe_DIR} depends neither on
  the sequence $\{\varepsilon_n\}$ nor on the subsequence
  $\{\varepsilon_{n_k}\}$, we conclude that
	\begin{equation}\label{conv_tWe_DIR-new}
\tW^E_{\varepsilon}\rightharpoonup\tW\quad\mbox{weakly in }
D^{m,2}_0(\R^N)\text{ as }\varepsilon\to0.
	\end{equation}
        \indent \textbf{Step 1 (Navier case
          $\boldsymbol{V^2=H^2_\vartheta}$).} We recall that here we
        are assuming $m=2$. The boundedness of
        $\|\Delta\tW_\varepsilon\|_{L^2(\varepsilon^{-1}\Omega)}$ with
        a constant independent of $\varepsilon$ follows  from the
        Dirichlet case, by recalling Proposition
        \ref{Capacity_monotonicity}(\textit{iii}). However, unlike the
        former case, one cannot now extend $\tW_\varepsilon$ to $0$
        outside $\varepsilon^{-1}\Omega$ and still obtain a function
        in $D^{2,2}_0(\R^N)$. To overcome this problem we rely on the
        Hardy-Rellich inequality proved in Theorem \ref{HR_ineq}. In
        fact, we have
	\begin{equation*}
		\int_{\varepsilon^{-1}\Omega}\frac{|\tW_\varepsilon|^2}{|x|^4}\,dx+\int_{\varepsilon^{-1}\Omega}\frac{|\nabla\tW_\varepsilon|^2}{|x|^2}\,dx\les\int_{\varepsilon^{-1}\Omega}|\Delta\tW_\varepsilon|^2\leq C
	\end{equation*}
	and therefore, by a diagonal process of extracted
        subsequences,
 for every sequence
  $\varepsilon_n\to0^+$ there exist a subsequence $\varepsilon_{n_j}$
  and $\tW\in H^2_{loc}(\R^N)$ for which
	\begin{equation}\label{conv_tWe_NAV}
		\frac{\nabla^{2-k}\tW_{\varepsilon_{n_j}}}{|x|^k}\rightharpoonup\frac{\nabla^{2-k}\tW}{|x|^k}\qquad\mbox{in}\ L^2(B_R)
	\end{equation}
	as $j\to\infty$ for any $R>0$ and $k\in\{0,1,2\}$. By
        weak lower semicontinuity of the norm, we infer that
	$$\int_{B_R}\frac{|\nabla^{2-k}\tW|^2}{|x|^{2k}}\,dx\leq\liminf_{j\to\infty}\int_{B_R}\frac{|\nabla^{2-k}\tW_{\varepsilon_{n_j}}|^2}{|x|^{2k}}\,dx\leq C,$$
	so that, letting $R\to+\infty$, 
	\begin{equation*}
	\int_{\R^N}\frac{|\nabla^{2-k}\tW|^2}{|x|^{2k}}\,dx\leq
          C\quad\text{for all }k\in\{0,1,2\}.
	\end{equation*}
	By Proposition \ref{HR_spaces}, this is equivalent to $\tW\in D^{2,2}_0(\R^N)$.
	
	It remains to prove that $\tW$ is the capacitary potential
        relative to $\capp_{2,\R^N}(\cK,U_0)$. Let $\eta$
          be as in the former case. Let $\varphi\in
          C^\infty_0(B_{1})$ be such that $\varphi\equiv1$ in $B_{1/2}(0)$ and consider the
        scaled functions $\varphi_R:=\varphi\big(\tfrac\cdot R\big)$
        with $R>r(M)$. Then
        $\varphi_R\big(\tW_{\varepsilon_{n_j}}-\eta
        U_{\varepsilon_{n_j}}\big)\rightharpoonup\varphi_R\big(\tW-\eta
        U_0\big)$ weakly in $H^2_0(B_R)$ and
        $\varphi_R(\tW_\varepsilon-\eta U_\varepsilon\big)\in
        H^2_0(B_R\setminus\eKe)$. By (M$_R$2-\textit{i}) we know then
        that
        $\varphi_R\big(\tW-\eta U_0\big)\in
        H^2_0(B_R\setminus\cK)$. Now we claim that
        $\varphi_R\big(\tW-\eta
          U_0\big)\to \tW-\eta U_0$ as $R\to+\infty$ in  $D^{2,2}_0(\R^N)$, thus
          concluding that $\tW-\eta U_0\in
        D^{2,2}_0(\R^N\setminus\cK)$. Indeed,
	\begin{equation*}
		\begin{split}
			\|\Delta\big(\left(\varphi_R-1\right)\big(\tW-\eta U_0\big)\big)\|_2^2&\les\|\Delta\varphi_R\big(\tW-\eta U_0\big)\|_2^2+\|\nabla\varphi_R\nabla\big(\tW-\eta U_0\big)\|_2^2\\
			&\quad+\|\left(\varphi_R-1\right)\Delta\big(\tW-\eta U_0\big)\|_2^2,
		\end{split}
	\end{equation*}
	where
	\begin{equation*}
		\|\left(\varphi_R-1\right)\Delta\big(\tW-\eta
                U_0\big)\|_2^2
                \leq\int_{\R^N\setminus B_{R/2}}|\Delta\big(\tW-\eta U_0\big)|^2\to 0
	\end{equation*}
	as $R\to+\infty$, and, for any $k\in\{1,2\}$,
	\begin{equation*}
		\begin{split}
			\|\nabla^k\varphi_R\nabla^{2-k}\big(\tW-\eta
                        U_0\big)\|_2^2&=\int_{\frac R2<|x|<R}\frac1{R^{2k}}\left|\big(\nabla^k\varphi\big)\left(\frac xR\right)\right|^2|\nabla^{2-k}\big(\tW-\eta U_0\big)|^2\,dx\\
			&\les \int_{\R^N\setminus B_R(0)}\frac{|\nabla^{2-k}\big(\tW-\eta U_0\big)|^2}{|x|^{2k}}\,dx\to0
		\end{split}
	\end{equation*}
	as $R\to+\infty$ since $\tW-\eta U_0\in D^{2,2}_0(\R^N)$ together with Proposition \ref{HR_spaces}.
	
	Next, we verify that
        \begin{equation}\label{eq:3}
          \int_{\R^N\setminus\cK}\Delta\tW\,\Delta\varphi=0\quad\text{for all
            $\varphi\in D^{2,2}_0(\R^N\setminus\cK)$}.
        \end{equation}
By density, it is enough to prove
        \eqref{eq:3} for all $\varphi \in
        C^\infty_0(\R^N\setminus\cK)$. Letting $\varphi \in
        C^\infty_0(\R^N\setminus\cK)$, there exist $R>r(M)$ and
        $\varepsilon_0>0$ such that 
        $\supp\,\varphi\subset B_R\subset\varepsilon^{-1}\Omega$ for all
        $\varepsilon<\varepsilon_0$, so that 
        $\varphi\in H^2_0(B_R\setminus\cK)$. By
        (M$_R$2-\textit{ii}) there exists a family
        $\{\varphi_\varepsilon\}_\varepsilon\subset H^2_0(B_R)$ such
        that $\varphi_\varepsilon\in H^2_0(B_R\setminus\eKe)$ and
        $\varphi_\varepsilon\to\varphi$ in $H^2_0(B_R)$. Hence,
	\begin{equation*}
          0=\int_{\varepsilon_{n_j}^{-1}\Omega}\Delta\tW_{\varepsilon_{n_j}}\Delta\varphi_{\varepsilon_{n_j}}=\int_{B_R}\Delta\tW_{\varepsilon_{n_j}}\Delta\varphi_{\varepsilon_{n_j}}\to\int_{B_R}\Delta\tW\,\Delta\varphi=\int_{\R^N\setminus\cK}\Delta\tW\,\Delta\varphi
	\end{equation*}
        as
          $j\to\infty$,
          by weak-strong convergence in $H^2_0(B_R)$.
We have thereby  proved the claim that $\tW$ is the capacitary
potential relative to $\capp_{2,\R^N}(\cK,U_0)$.
Since the limit $\tW$ in \eqref{conv_tWe_NAV} depends neither on
  the sequence $\{\varepsilon_n\}$ nor on the subsequence
  $\{\varepsilon_{n_k}\}$, we conclude that
the convergences in  \eqref{conv_tWe_NAV} actually hold as
$\varepsilon\to 0$, i.e.
	\begin{equation}\label{eq:4}
		\frac{\nabla^{2-k}\tW_{\varepsilon}}{|x|^k}\rightharpoonup\frac{\nabla^{2-k}\tW}{|x|^k}\qquad\mbox{in}\
                L^2(B_R)\quad\text{as }\varepsilon\to0\quad\text{for
                  all $R>0$ and $k\in\{0,1,2\}$}.
	\end{equation}

\vskip0.2truecm \indent \textbf{Step 2.} ($m=2$ in the Navier case,
$m\geq2$ in the Dirichlet). We aim now to prove the asymptotic expansions
\eqref{Thm_blowup_asympt}--\eqref{Thm_blowup_asympt_Nav}. As above, let $\eta\in C^\infty_0(\R^N)$ be equal to $1$ on an
open set $\mathcal U$ with $\cK\cup M\subset\mathcal U$.
Let $R>0$ be such that $\supp\,\eta\subset B_R$.
Since
$\tW_\varepsilon-\eta U_\varepsilon\in
V^m_0(\varepsilon^{-1}\Omega\setminus\eKe)$, by \eqref{tW} and \eqref{cap_Rn}, together with
\eqref{conv_tWe_DIR-new} or \eqref{eq:4}, we obtain that
	\begin{equation}\label{cap_Rn_convergence}
		\begin{split}
			\|\nabla^m\tW_\varepsilon\|^2_{L^2(\varepsilon^{-1}\Omega)}&=\int_{\varepsilon^{-1}\Omega}\nabla^m\tW_\varepsilon\,\nabla^m\left(\eta U_\varepsilon\right)=\int_{B_R}\nabla^m\tW_\varepsilon\,\nabla^m\!\left(\eta U_\varepsilon\right)\\
			&\to\int_{\R^N}\nabla^m\tW\,\nabla^m\!\left(\eta U_0\right)=\|\nabla^m\tW\|^2_{L^2(\R^N)}
		\end{split}
	\end{equation}
	as $\varepsilon\to0$ by weak-strong convergence. On the other
        hand, by rescaling one has that
	\begin{equation}\label{cap_rescaled}
		\begin{split}
			\|\nabla^m\tW_\varepsilon\|^2_{L^2(\varepsilon^{-1}\Omega)}&=\frac1{\varepsilon^{2\gamma}}\int_{\varepsilon^{-1}\Omega}\left|\nabla^m\!\left(W_\varepsilon(\varepsilon x)\right)\right|^2\,dx=\varepsilon^{-N+2m-2\gamma}\intOmega|\nabla^m W_\varepsilon(y)|^2\,dy\\
			&=\varepsilon^{-N+2m-2\gamma}\capVm(K_\varepsilon,u_J).
		\end{split}
	\end{equation}
	Hence, from \eqref{cap_Rn} and \eqref{cap_Rn_convergence}-\eqref{cap_rescaled} we finally infer
that	\begin{equation*}
          \capVm(K_\varepsilon,u_J)=
          \varepsilon^{N-2m+2\gamma}	\|\nabla^m\tW_\varepsilon\|^2_{L^2(\varepsilon^{-1}\Omega)}=\varepsilon^{N-2m+2\gamma}\left(\capp_{m,\R^N}(\cK,U_0)+{\scriptstyle \mathcal{O}}(1)\right)
	\end{equation*}
        as $\varepsilon\to0$.
\end{proof}

\subsection{Sufficient conditions for a sharp asymptotic expansion}\label{Sec_suffcond}
Looking at the asymptotic expansions we have found in Theorems
\ref{asymptotic_exp_eigv_cap}-\ref{asymptotic_exp_eigv_cap_Nav},  one
may ask whether the results are sharp, in the sense that the vanishing
rate of the eigenvalue variation $\lambda_J(\Omega\setminus
K_\varepsilon)-~\!\lambda_J(\Omega)$ is equal to
$N-2m+2\gamma$. The next results provide sufficient conditions on
$\cK$ and $U_0$ in order to ensure that  $\capp_{m,\R^N}(\cK,U_0)\neq0$.
\begin{prop}\label{sharp_asympt}
  Under the assumptions of Theorems \ref{asymptotic_exp_eigv_cap} or
  \ref{asymptotic_exp_eigv_cap_Nav}, suppose that the Lebesgue measure
  of $\cK$ is positive. Then
	\begin{equation*}
		\lim_{\varepsilon\to0}\frac{\lambda_J(\Omega\setminus K_\varepsilon)-\lambda_J(\Omega)}{\varepsilon^{N-2m+2\gamma}}=\capp_{m,\R^N}(\cK,U_0)>0.
	\end{equation*}
\end{prop}
\begin{proof}
  Denote by $W^{(0)}_{\cK}$ the capacitary potential for
  $\capp_{m,\R^N}(\cK,U_0)$ and suppose by contradiction that
  $\capp_{m,\R^N}(\cK,U_0)=0$. Then, by the Hardy inequality
  for the polyharmonic operator (see \cite[Theorem 12]{DH}), there
  exists a constant $c=c(N,m)$ such that
	\begin{equation*}
		0=\|\nabla^mW^{(0)}_{\cK}\|_{L^2(\R^N)}^2\geq c\int_{\R^N}\frac{|W^{(0)}_{\cK}|^2}{|x|^{2m}}\,dx\geq c\int_{\cK}\frac{|W^{(0)}_{\cK}|^2}{|x|^{2m}}\,dx= c\int_{\cK}\frac{|U_0|^2}{|x|^{2m}}\,dx
	\end{equation*}
	since $W^{(0)}_{\cK}\equiv U_0$ on $\cK$. Since $|\cK|>0$,
        this readily implies that $U_0$ vanishes a.e. on $\cK$. However, by construction, $U_0$ is a polyharmonic polynomial on $\R^N$ which is not identically zero (see \cite[Sec.4 Theorem 1]{Bers}), so 
	it cannot vanish on a set of positive
        measure (since nontrivial analytic functions cannot
          vanish on positive measure sets). This, together with
          Theorems \ref{asymptotic_exp_eigv_cap} and 
  \ref{asymptotic_exp_eigv_cap_Nav}, concludes the proof.
\end{proof}

The next results apply to some specific situations in which, although $K$ has vanishing Lebesgue measure, one may anyway have that $\capp_{m,\R^N}(\cK,U_0)\neq0$. 

\begin{prop}\label{suff_cond_nonvanishing}
	Let $N>2m$ and $\cK\subset\R^N$ be a compactum with
        $\capp_{m,\R^N}(\cK)>0$. Suppose moreover that
        $u_J(0)\neq0$. Then, in the setting of Theorems
        \ref{asymptotic_exp_eigv_cap} or
        \ref{asymptotic_exp_eigv_cap_Nav}, we have that
	\begin{equation}\label{expansion_eigv_gamma0}
		\lambda_J(\Omega\setminus K_\varepsilon)=\lambda_J(\Omega)+\varepsilon^{N-2m}u_J^2(0)\capp_{m,\R^N}(\cK)+{\scriptstyle \mathcal{O}}(\varepsilon^{N-2m})
	\end{equation}
	as $\varepsilon\to0$.
\end{prop}
We mention that an expansion of type
  \eqref{expansion_eigv_gamma0} was obtained in \cite[Theorem
  1.4]{Courtois} and \cite[Theorem 1.7]{AFHL} in the case $N=2m=2$, in
  which the vanishing rate of the eigenvalue variation is logarithmic.
\begin{proof}[Proof of Proposition \ref{suff_cond_nonvanishing}]
	Since $\{K_\varepsilon\}_{\varepsilon>0}$ is concentrating at $\{0\}$ and $u_J(0)\neq0$, then the degree $\gamma$ of the polynomial $U_0$ is $0$, and $U_0=u_J(0)$. It is then easy to see that
	$$\capp_{m,\R^N}(\cK,u_J(0))=u_J^2(0)\capp_{m,\R^N}(\cK)>0,$$
        so that \eqref{expansion_eigv} and \eqref{expansion_eigv_Nav}
        can be rewritten as in \eqref{expansion_eigv_gamma0}.
\end{proof}

In the case $u_J(0)=0$, the next result, inspired by \cite[Lemma
3.11]{FNO}, may be useful. It tells that, if the
compactum $\cK$ and the null-set of the polynomial $U_0$ are
``transversal enough'', then again $\capp_{m,\R^N}(\cK,U_0)>0$.
\begin{prop}\label{Prop_transv}
  Let $N>2m$ and $\cK\subset\R^N$ be a compactum with
  $\capp_{m,\R^N}(\cK)>0$. Letting $f\in~\!\!C^\infty(\R^N)$, let us
  consider the set $Z_f^{\cK}:=\{x\in \cK\,|\,f(x)=0\}$. If
  $\capp_{m,\R^N}(Z_f^{\cK})<\capp_{m,\R^N}(\cK)$, then
  $\capp_{m,\R^N}(\cK,f)>0$.
\end{prop}
\begin{proof}
  Let $\{\cU_n\}$ be a sequence of nested open sets in $\R^N$ so that
  $Z_f^{\cK}\subset\cU_n$ for all $n\in \N$ and 
  $Z_f^{\cK}=\bigcap_{n\in\N}\overline{\cU_n}$ and let
  $\cK_n:=\cK\setminus\cU_n$, which is a sequence of compact sets. By
  subadditivity and monotonicity of the capacity (see
    e.g. \cite{M}) one has that 
\begin{equation*}
		\capp_{m,\R^N}(\cK)\leq\capp_{m,\R^N}(\cK_n)+\capp_{m,\R^N}(\overline{\cU_n}).
	\end{equation*}
	Moreover, fixing
        $0<\delta<\capp_{m,\R^N}(\cK)-\capp_{m,\R^N}(Z_f^{\cK})$, one
        may find a neighbourhood $\cU(Z_f^{\cK})$ such that one has
        $\overline{\cU_n}\subset\cU(Z_f^{\cK})$ by construction and
        $\capp_{m,\R^n}(\overline{\cU_n})\leq\capp_{m,\R^n}(Z_f^{\cK})+\delta$
        by Lemma \ref{capacity_right_continuity}, provided $n$ is
        large enough. This implies
	\begin{equation}\label{capp_Kn+_pos}
		\capp_{m,\R^N}(\cK_n)\geq\capp_{m,\R^N}(\cK)-\capp_{m,\R^N}(Z_f^{\cK})-\delta>0
	\end{equation}
	for $n$ large enough. We define 
        $\cK_n^+:=\{ x\in\cK_n:f(x)>0\}$ and 
        $\cK_n^-:=\{ x\in\cK_n:f(x)<0\}$ for all $n\in\N$.  Noticing that $\cK_n$
        is the union of $\cK_n^+$ and $\cK_n^-$, necessarily
either $\capp_{m,\R^N}(\cK_n^+)>0$ or
        $\capp_{m,\R^N}(\cK_n^-)>0$; let us e.g. consider the case 
        $\capp_{m,\R^N}(\cK_n^+)>0$. By regularity of $f$ and
        since $\cK_n^+$ is compact, then $c_n^+:=\inf_{\cK_n^+}f$ is
        attained and strictly positive. Take now any
        $u_n\in D^{m,2}_0(\R^N)$ so that
        $u_n-\eta_{\cK_n^+}f\in D^{m,2}_0(\R^N\setminus\cK_n^+)$ and
        define $v_n:=\frac{u_n}{c_n^+}$. Then it is clear that
        $v_n\in D^{m,2}_0(\R^N)$ and $v_n\geq1$ a.e. on
        $\cK_n^+$. Hence,
	\begin{equation*}
          \Capp(\cK_n^+)\leq\int_{\R^N}|\nabla^mv_n|^2
          =\frac1{\left(c_n^+\right)^2}\int_{\R^N}|\nabla^mu_n|^2.
	\end{equation*}
	By arbitrariness of $u_n$ this yields
        $\left(c_n^+\right)^2\Capp(\cK_n^+)\leq
        \capp_{m,\R^N}(\cK_n^+,f)\leq\capp_{m,\R^N}(\cK,f)$, since
        $\cK_n^+\subset \cK$ for all $n\in\N$. Using now the
        equivalence of the capacities in $\R^N$ stated in Lemma
        \ref{Lemma_Mazya_equivalence}, one infers that
	\begin{equation*}
		\capp_{m,\R^N}(\cK,f)\geq c\left(c_n^+\right)^2\capp_{m,\R^N}(\cK_n^+)>0
	\end{equation*}
	by \eqref{capp_Kn+_pos}. This concludes the proof.
\end{proof}
\begin{remark}
  In view of Remark \ref{Remark_manifolds} it is immediate to see
  that, if $\capp_{m,\R^N}(\cK)>0$ and $Z_{U_0}^{\cK}$ has dimension
  $d\leq N-2m$, then the assumptions of Proposition \ref{Prop_transv}
  are fulfilled, thus ensuring that $\capp_{m,\R^N}(\cK,U_0)>0$.
\end{remark}

\section{Open problems}\label{Sec_OP}
We finally discuss possible generalizations and questions which are left open by our analysis and which we believe of interest.

\paragraph{Higher-order Navier setting.} The results in Section
\ref{Sec_conv_eigv} for the Navier setting are obtained in the general
case $m\geq2$. On the other hand, Theorem \ref{Thm_blowup_Nav} and its
consequent Theorem \ref{asymptotic_exp_eigv_cap_Nav} are established
only for $m=2$. The main difficulty in their extension to higher
  orders relies in the characterization of homogeneous Sobolev
spaces via Hardy-Rellich inequalities. In our argument this was
  necessary to compensate for the lack of a trivial extension, which
  is instead available in the
Dirichlet setting. Although we envision that a generalization of the
Hardy-Rellich inequality of Proposition \ref{HR_ineq} is reachable,
the extension of the characterization
contained in Proposition \ref{HR_spaces} seems to be a non
  trivial problem. Indeed, for $m=2$ the only intermediate derivative is the gradient and
$\nabla u=Du$; on the other hand for $m\geq3$ the Hardy-Rellich
inequality would provide a weighted estimate on the derivatives
$\nabla^ku$, $k\in\{1,\dots,m-1\}$, while one would need to estimate
the full tensor of the derivatives $D^ku$ to be able to conclude that
$u\in
D^{m,2}_0(\R^N)$.

\paragraph{Small dimensions.} Most of our results deal with the high
dimensional case $N\geq2m$, because the concentration of the family of
sets $\{K_\varepsilon\}_{\varepsilon>0}$ to a zero $V^m$-capacity
compact set was needed. Recall that for $N<2m$ all compact sets are of
positive capacity, see Proposition \ref{Capacity_point}.
Nevertheless, in order to prove that the asymptotic expansions given
by Theorem \ref{Prop_eigv_cap} are sharp, in Theorems
\ref{Thm_blowup_easycase} and \ref{Thm_blowup_easycase_Nav} we have to
restrict to $N>2m$. The conformal case $N=2m$ seems not to be tratable
with the blow-up analysis, not only due to the different
characterization of the spaces $D^{m,2}_0(\R^N)$ and the use of
Hardy-Rellich inequalities, but also because the $m$-capacity
$\capp_{m,\R^{2m}}(K)$ of any compact set $K$ in $\R^{2m}$ is null
(see \cite {M1}). A
different approach for conformal (and smaller!)  dimensions should be
in fact developed and we expect that the expansion involves the
logarithm of the diameter of the shrinking sets, in analogy with the
results in \cite{AFHL} for the case $m=1$. We remark in particular
that, when our results are applied to the biharmonic operator,
i.e. $m=2$, we cover the case $N\geq5$, while the two-dimensional
case, from a completely different point of view, is studied in
\cite{CN,KLW,LWK}. The case $N=3$ is left open and $N=4$ only
partially answered by Theorem \ref{Prop_eigv_cap}.

\paragraph{Equivalent definitions of capacities.} As described in
Section \ref{Section_cap}, both $\capp_{m,\R^N}$ in \eqref{capp_RN}
and $\Capp$ in \eqref{Capp} are good definitions of capacity, and they
are also equivalent for $N>2m$, see Lemma
\ref{Lemma_Mazya_equivalence}. In the second order case, it is not
difficult to prove that the two coincide, while - up to our knowledge
- this is still unknown in the higher-order setting. A weak question
would be to ask whether the two capacities are asymptotic for families
of shrinking domains, e.g. for $K_\varepsilon=\varepsilon\cK$ as
considered in Sec. \ref{Sec_blowup}. It would be also interesting to
understand whether the equivalence remains true for the weighted
capacities $\capp_{m,\R^N}(\cdot,h)$ and the analogue $\Capp(\cdot,h)$
for some class of nonconstant functions $h$.
                 
\paragraph{Boundary conditions.} As mentioned in the introduction, it
would be interesting to investigate the complementary cases of
prescribing Navier BCs on the removed set and either Navier or
Dirichlet BCs on the external boundary $\dOmega$. Because of the lack
of an extension by zero in case of Navier BCs, which has consequences
on the mutual relations between the spaces $V^m(\Omega\setminus K)$, a
different argument would be needed. More in general, it would be
challenging to consider more general types of BCs, which yield a
different quadratic form associated to the polyharmonic operator,
which would involve also boundary integrals. An interesting case in
the biharmonic setting, related to the physical model of thin plates,
is for example given by Steklov BCs $u=\Delta u-d\kappa\partial_nu=0$,
$d\in\R$ and $\kappa$ being the signed curvature of the
boundary. 

\section*{Acknowledgements}

V. Felli is partially supported by the INdAM-GNAMPA 2022 project
``Questioni di esistenza e unicit\`a per problemi non locali con
potenziali''. The main part of this work was carried out while G. Romani was
supported by a postdoctoral fellowship at the University of Milano-Bicocca.

\end{document}